\documentclass[11pt]{article}
\usepackage{latexsym,color,amsmath,amsthm,amssymb,amscd,amsfonts}

\newtheorem{theo}{Theorem}
\newtheorem{cor}{Corollary}
\newtheorem{lem}{Lemma}
\newtheorem{defn}{Definition}
\newtheorem{prop}{Proposition}
\theoremstyle{remark}

\providecommand{\sca}[1]{\langle #1 \rangle}

\def\E{\mathcal{E}}
\newcommand{\ps}{\psi^{\star}_{(s)}}
\def\Fs{\mathcal{F}^+_{s, \star}}
\def\Ms{\mathcal{L}_{s, \star}}
\def\C{\mathcal{C}}

\def\OrliczGmix{G_{\vec{h}, \vec F^1, \vec F^2}^{orlicz}}
\def\OrliczAmix{as_{\vec{h}, \vec F^1, \vec F^2}^{orlicz}}

\def\OrliczAmixith{as_{\vec h, i,  \vec F^1, \vec F^2}^{orlicz}}
\def\OrliczGmixith{G_{\vec h, i,  \vec F^1, \vec F^2}^{orlicz}}

\def\OrliczAmixjth{as_{\vec h, j,  \vec F^1, \vec F^2}^{orlicz}}
\def\OrliczAmixkth{as_{\vec h, k,  \vec F^1, \vec F^2}^{orlicz}}

\def\OrliczAmixrk{as_{\vec{h}_{r,k}, \vec{F}^1_{r,k}, \vec{F}^2_{r,k}}^{orlicz}}
\def\OrliczGmixrk{G_{\vec{h}_{r,k}, \vec{F}^1_{r,k}, \vec{F}^2_{r,k}}^{orlicz}}

\def\OrliczG{G_{h, F_1, F_2}^{orlicz}}
\def\OrliczA{as_{h,F_1, F_2}^{orlicz}}

\def\OrliczAk{as_{h_k, F^1_k, F^2_k}^{orlicz}}
\def\OrliczGk{G_{h_k, F^1_k, F^2_k}^{orlicz}}

\def\OrliczAi{as_{h_i, \breve{F_i}, \breve{F_i}}^{orlicz}}
\def\OrliczGi{G_{h_i, \breve{F_i}, \breve{F_i}}^{orlicz}}

\def\F{{\cal F}^+_{\psi^*}}

\def\R{{\mathbb R}}

 \def\bbR{{\mathbb R}}
 \def\Cs{\mathcal{C}_s}

\def\lam{\lambda}

\def\to{\rightarrow}

\def\pmx{\begin{pmatrix}}
\def\emx{\end{pmatrix}}
\def\Hess{\nabla^2}

\def\det{{\rm det}}

 \def\M{\mathcal{L}_{\psi^*}}

\def\R{\mathbb R}

\def\OrliczAS{as_{h, s}^{orlicz}}
\def\OrliczGS{G_{h, s}^{orlicz}}

\begin{document}
\title{Affine isoperimetric inequalities in the functional Orlicz-Brunn-Minkowski theory
}

\author{Umut Caglar and Deping Ye }
\date{}
\maketitle
\begin{abstract} In this paper, we develop a basic theory of Orlicz affine and geominimal surface areas for convex and $s$-concave functions.   We prove some basic properties for these newly introduced functional affine invariants and establish related functional affine isoperimetric inequalities as well as  functional Santal\'o type inequalities.  

\end{abstract}

   \section{Introduction} 
   
 The definition of Orlicz addition by Gardner, Hug and Weil \cite{Gardner2014} and Xi, Jin and Leng \cite{XJL} brings new impulses to the rapidly developing Orlicz-Brunn-Minkowski theory for convex bodies. In fact, Orlicz addition makes it possible to establish the Orlicz-Brunn-Minkowski inequality, develop Orlicz mixed volume,
 and prove the Orlicz-Minkowski inequality for the Orlicz mixed volume. However, the first steps in this theory were actually the Orlicz affine isoperimetric inequalities for Orlicz centroid bodies and Orlicz projection bodies by Lutwak, Yang, and Zhang \cite{LYZ2010a, LYZ2010b}.  An affine isoperimetric inequality in the Orlicz-Brunn-Minkowski theory provides upper and/or lower bounds,  in terms of volume,  for functionals defined on convex bodies which are invariant under all volume preserving linear transforms; and it would be ideal if these functionals attain their maximum or minimum at (and only at) ellipsoids. It is convenient and natural to call affine isoperimetric inequalities in the Orlicz-Brunn-Minkowski theory as Orlicz affine isoperimetric inequalities,  just like the $L_p$ affine isoperimetric inequalities in the $L_p$-Brunn-Minkowski theory.  Another example of Orlicz affine isoperimetric inequalities is the one by the second author \cite{Ye2014a2}, which provides bounds for Orlicz affine and geominimal surface areas, that is,  under certain conditions,  Orlicz affine and geominimal surface areas attain their maximum (or minimum) at and only at ellipsoids.    

Developing and extending affine surface areas has been a central goal in convex geometry for decades. The following are the major steps. The first major step was due to Blaschke \cite{Bl1}, who defined the classical $L_1$ affine surface area.  Then, Lutwak \cite{Lutwak1996} introduced $L_p$ affine surface areas for $p>1$. Based on some beautiful integral formulas for $L_p$ affine surface areas (which essentially involve Gauss curvature and the support function), Sch\"{u}tt and Werner \cite{SW2004} proposed a further extension of $L_p$ affine surface area to $-n\neq p\in \bbR$. Later, Ludwig and Reitzner \cite{LudR} and Ludwig  \cite{Ludwig2009} introduced the general affine surface areas for non-homogeneous functions. Note that the above affine surface areas are not continuous with respect to the Hausdorff metric. However, the classical $L_1$ geominimal surface area, which is closely related to the classical $L_1$ affine surface area, was proved to be continuous with respect to the Hausdorff metric and to be a bridge between several different type of geometries (see Petty \cite{Petty1974} for more details).  Since there are no convenient integral formulas for $L_p$ geominimal
surface areas for $p > 1$, for the definition of the $L_p$ geominimal surface area for $-n\neq p\in \bbR$, a different approach from those used in  \cite{Ludwig2009, LudR, SW2004} is needed; and that was proposed in \cite{Ye2014a1} (actually, such an approach was motivated by Lutwak's definition of the $L_p$ geominimal surface area for $p>1$ \cite{Lutwak1996} and the work \cite{Xiao2010} by Xiao). In fact, the approach in \cite{Ye2014a1} also provides alternative definitions for the $L_p$ affine surface areas for $-n\neq p \in \bbR$. This opens the door to develop Orlicz affine and geominimal surface areas \cite{Ye2014a2}, as well as their duals for star bodies \cite{Ye2015b2} (based on the dual Orlicz mixed volume in \cite{GardnerYe2015}) and their mixed  counterparts involving multiple convex bodies \cite{Ye2014a2, Ye2015a}. See e.g., \cite{Ludwig2009, Lutwak1996, Petty1985, WernerYe2008,  Ye2013, Ye2014a2, Ye2014a1} for affine isoperimetric inequalities related to affine and geominimal surface areas.

The geometry of log-concave functions aims to study the geometric properties of log-concave functions, in a manner similar to the geometry of convex bodies (also known as convex geometry or the Brunn-Minkowski theory of convex bodies).  In fact, there is a ``dictionary" between these two theories, for instance, integral translates to volume, log-concave functions to convex bodies, the Gaussian function 
$e^{-\frac{\|\cdot\|^2}{2}}$ to the unit Euclidean ball, polar duals of log-concave functions to polars of convex bodies, and the integral product to the Mahler volume product. The geometry of log-concave functions extends fundamental notions and results in convex geometry nontrivially to their functional counterparts. Moreover, it usually provides much more powerful tools and far-reaching results than its geometric counterpart (indeed,  every convex body can be associated with a log-concave function). See, e.g.,  Klartag and Milman \cite{KlartagMilman} and Milman \cite{VMilman2008} for more detailed motivation and references.

An important functional affine isoperimetric inequality is the functional  Blaschke-Santal\'o inequality  \cite{ArtKlarMil, KBallthesis,  FradeliziMeyer2007, Lehec2009, Lehec2009b}, which is essential for the isoperimetric inequalities for $L_p$ affine surface areas of log-concave and s-concave functions \cite{Caglar-6, CaglarWerner, CaglarWerner2}. In their seminal paper \cite{ArtKlarSchuWer},  Artstein-Avidan,  Klartag, Sch\"{u}tt and Werner provided a definition of $L_1$ affine surface area for $s$-concave functions and established  related functional affine isoperimetric inequality. In particular,  a functional affine isoperimetric inequality for log-concave
functions was given and can be viewed as an inverse logarithmic Sobolev  inequality for entropy. These inequalities further imply a version of the reverse Poincar\'{e} inequality \cite{ArtKlarSchuWer}. The main purpose of this paper is to develop a theory of Orlicz affine and geominimal surface areas for convex functions (hence also for log-concave functions) as well as  their related functional affine isoperimetric inequalities. The results in this paper bring more items into the above mentioned ``dictionary" and hopefully will provide powerful tools for many related fields, such as, analysis, (convex) geometry, and information theory.

This paper is organized as follows. In Section \ref{section 3}, we give a new formula for a general functional $L_p$ affine surface area for convex functions. Then, we generalize this idea and introduce the general Orlicz affine and geominimal surface areas for convex functions. We prove that these new concepts are  $SL_{\pm}(n)$-invariant.  We also prove some inequalities for these notions, such as functional affine isoperimetric inequalities, and generalizations of functional Blaschke-Santal\'o and inverse Santal\'o inequalities.
In Section \ref{section 4}, we propose the definition of Orlicz affine and geominimal surface areas for $s$-concave functions and prove corresponding functional inequalities, e.g., functional affine isoperimetric and Santal\'o type inequalities. In Section \ref{section:mixed}, we will briefly discuss results for multiple convex functions.  

 
\section{The general Orlicz affine and geominimal  surface areas for convex functions} \label{section 3} 

 Let $(\bbR^n, \|\cdot\|)$ be the Euclidean space with  $\|\cdot\|$ the Euclidean norm of  $\R^n$ induced by the usual inner product $\langle \cdot, \cdot \rangle$.  Let  $\C$ be the set of all convex functions $\psi: \R^n \to \R \cup \{+\infty\}.$  Throughout this paper, the interior of the convex domain of $\psi\in \C$ is always assumed to be nonempty. Denote by $\psi^*$ the classical Legendre transform of $\psi$, that is,  \begin{equation}
\label{Legendre}
\psi^* ( y ) = \sup_{x\in \bbR^n}  \bigl( \sca{x,y} - \psi (x) \bigr) .
\end{equation}   Clearly, $\psi ( x) + \psi^* (y ) \geq \sca{x,y}$ for all $x,y \in \R^n$. Equality holds if and only if
 $x$ is in the domain of $\psi$ and $y$ is in the subdifferential 
 of $\psi$ at $x$:  for almost all $x$ in the domain of $\psi\in \C$,   \begin{equation*} 
 \psi^* ( \nabla \psi (x) ) = \langle x ,\nabla \psi (x) \rangle - \psi (x),\end{equation*} where $\nabla \psi$ denotes the gradient of $\psi$. Rademacher's theorem (e.g., \cite{Rademacher}) asserts that $\nabla \psi$ exists almost everywhere.  For $\psi\in \C$,  $\nabla^2\psi$ denotes the Hessian matrix of $\psi$  in the sense of  Alexandrov, and it exists almost everywhere by a theorem of Alexandrov \cite{Alexandroff}  and Busemann-Feller  \cite{Buse-Feller}.  Let $$X_\psi=\Big\{x\in \bbR^n:\ \psi(x)<\infty,\ \mathrm{and}\  \nabla^2\psi(x)\ \mathrm{exists\  and\  is\ invertible}\Big\}.$$ For more background on convex functions, please see \cite{mccann, Rockafellar, SchneiderBook}.

  Denote by $f^\circ$  the polar dual of  the function $f: \bbR^n\rightarrow [0, \infty)$, which has the form:  $$f^\circ(x)=\inf _{y\in \R^n} \Big(\frac{e^{-\langle x, y\rangle}}{f(y)}\Big)\ \  \Leftrightarrow \ \  -\log f^\circ =(-\log f)^*. $$
A function $f: \bbR^n\rightarrow [0, \infty)$ is log-concave if $\log f$ is concave on the support of $f$. Note that $f^\circ$ is always a log-concave function no matter whether $f$ is log-concave or not.  A log-concave function $f$ is often written as $f=e^{-\psi}$ with  $\psi\in \C$, and clearly $f^\circ=e^{-\psi^*}$. Moreover, $(f^\circ)^\circ=f$ if $f$ is an upper semi-continuous log-concave function. The function $\gamma_n=e^{-\frac{\|\cdot\|^2}{2}}$ serves as  the ``unit Euclidean ball"  of log-concave functions as $(\gamma_n)^\circ=\gamma_n$, and its integral over $\bbR^n$ is equal to $(\sqrt{2\pi})^n$.  

   Throughout the paper, we always assume that the functions we consider, such as $F_1, F_2 \colon \R\to (0, \infty)$ and $\psi\in \C$, have enough smoothness and integrability to guarantee the integrals or other expression well-defined. For instance, we will need the following integrals to be finite $$0<\int_{X_{\psi}} F_1(\psi(x))\,dx<\infty \ \  \mathrm{and} \ \ \ 0< \int_{X_{\psi^*}} F_2(\psi^*(y)) dy<\infty.$$ 

 \subsection{A new formula for a general  $L_p$ affine surface area for convex functions}\label{section 3.1}  

 The following general  $L_p$ affine surface area for convex functions was proposed in   \cite{Caglar-6}.  
\begin{defn}
\label{def:log}
For  measurable functions  $F_1, F_2 \colon \R\to (0, \infty)$, $-n\neq p\in\R$, and $\psi\in \C$,  define
\begin{equation}\label{general}
as_{p, F_1, F_2}(\psi) = \int_{X_\psi} \big(F_1 (\psi(x))\big)^{\frac{n}{n+p}} \big(F_2(\langle x, \nabla \psi(x)\rangle - \psi(x)) \ \det \,  \nabla^2  \psi(x)\big)^{\frac{p}{n+p}} \,dx.
\end{equation} 
\end{defn}
  \noindent {\bf Remark.} Note that $as_{p, F_1, F_2}(\cdot)$ is called  the general $L_p$ affine surface area  because the above definition is just the definition of the functional $L_p$ affine surface area for log-concave functions if $F_1 (t)= F_2(t) = e^{-t}$. Hence, functions $F_1$ and $F_2$ act like parameters and provide the power to include much wider class of functions than the log-concave functions.

 Denote  by $\F$ the set of all positive Lebesgue integrable functions defined on $X_{\psi^*}$. That is, $g\in \F$ if $g(y)>0$ for all $y\in X_{\psi^*}$ and  $0<I(g, \psi^*)<\infty$ with   \begin{eqnarray}  I(g, \psi^*)=\int_{X_{\psi^*}} g(y) \, dy=\int_{X_\psi}   g(\nabla \psi(x)) \det \,  \nabla^2  \psi(x) \, dx, \label{ratio of g} \end{eqnarray} where the  second equality follows from Corollary 4.3 and Proposition A.1 in \cite{mccann}. In particular,  \begin{eqnarray*}
  I(F_2\circ\psi^*, \psi^*)&=&\int_{X_{\psi^*}} F_2(\psi^*(y)) dy\\ &=& 
  \int_{X_\psi} F_2(\psi^*(\nabla \psi (x))) \det \Hess \psi(x) dx \nonumber\\ &=& \int_{X_\psi} F_2(\langle x, \nabla \psi (x)\rangle-\psi(x)) \det \Hess \psi(x) dx. 
  \end{eqnarray*}     We often need  $I(F_1\circ\psi, \psi) = \int_{X_{\psi}} F_1(\psi(x))\,dx$.    
 
 For  measurable functions $F_1, F_2 : \R\rightarrow (0, \infty)$, let \begin{eqnarray} V_{p, F_1, F_2} (\psi, g) =\int_{X_\psi} \left(\frac{F_2( \langle x, \nabla \psi(x)\rangle - \psi(x)) }{g(\nabla \psi(x))}\right)^{p/n} F_1(\psi(x))\,dx. \label{mixed:lp-1}\end{eqnarray}

The following theorem gives a new formula for the above general functional $L_{p}$ affine surface area. 
\begin{theo}
  \label{equivalent:affine:surface:area}   Let $\psi$ be a $C^2$ strictly convex function.   For $p\geq 0$,  one has \begin{equation*} as_{p, F_1, F_2}(\psi) =\inf_{ g\in \F } \left\{ V_{p, F_1, F_2} (\psi, g)^{\frac{n}{n+p}}\
I(g, \psi^*)^{\frac{p}{n+p}}\right\}, 
\end{equation*} while for $-n\neq p<0$, the above formula holds with `` $\inf$" replaced by `` $\sup$".  \end{theo}  

 \begin{proof}  We only prove the desired result for $p\in (0, \infty)$. The result for $-n\neq p<0$ follows along the same lines and for $p=0$ holds trivially.

  As $\psi$ is a $C^2$ strictly convex function, then $\det \nabla^2 \psi(x)>0$ on $X_{\psi}$  and $\nabla \psi: X_{\psi}\rightarrow X_{\psi^*}$ is smooth and bijective.  Consider the following function  $$g_0(\nabla \psi(x))=\big[F_2(\langle x, \nabla \psi(x)\rangle - \psi(x)) \big]^{\frac{p}{n+p}}  \bigg(\frac{F_1 (\psi(x))}{ \det \,  \nabla^2  \psi(x)}\bigg)^{\frac{n}{n+p}} \ \ \ \ \ 
 \mathrm{for}\  x\in X_{\psi}. $$ By formulas (\ref{general}), (\ref{ratio of g}) and (\ref{mixed:lp-1}), one can check $$as_{p, F_1, F_2}(\psi) = \big[V_{p, F_1, F_2} (\psi, g_0)\big]^{\frac{n}{n+p}} I(g_0, \psi^*)^{\frac{p}{n+p}}\geq \inf_{g\in \F} \  [V_{p, F_1, F_2} (\psi, g)]^{\frac{n}{n+p}} I(g, \psi^*)^{\frac{p}{n+p}}.$$   On the other hand, H\"{o}lder's inequality implies that for all  $g\in \F$,     
 \begin{eqnarray} as_{p, F_1, F_2}(\psi)   \!\! &=&\!\!  \int_{X_\psi}\!\! \bigg[F_1 (\psi(x))  \bigg(\frac{F_2(\langle x, \nabla \psi(x)\rangle - \psi(x)) }{g(\nabla \psi(x))}\bigg)^{\frac{p}{n}}\bigg]^{\frac{n}{n+p}} \big (g(\nabla \psi(x)) \det \,  \nabla^2  \psi(x)\big)^{\frac{p}{n+p}} dx \nonumber\\  \!\! &\leq &\!\!  
   \big[V_{p, F_1, F_2} (\psi, g)\big]^{\frac{n}{n+p}} I(g, \psi^*)^{\frac{p}{n+p}}. \label{affine:Holder} 
\end{eqnarray}  Taking the infimum over all  $g\in \F$, one gets, for $p\in (0, \infty)$,  
 \begin{equation*} as_{p, F_1, F_2}(\psi)   \leq \inf_{g\in \F} \  [V_{p, F_1, F_2} (\psi, g)]^{\frac{n}{n+p}} I(g, \psi^*)^{\frac{p}{n+p}},\end{equation*}  and hence the desired result holds.  \end{proof}
  
  \noindent {\bf Remark.} Let  $y=\nabla \psi(x)$, then    $\psi(x)+\psi^*(y)=\langle x, y\rangle$, $x=\nabla\psi^*(y)$ and  $\nabla ^2 \psi (x) \nabla ^2 \psi^*(y)=\mathrm{Id}$  (the identity matrix on $\bbR^n$).  
 These lead to the explicit expression of  $g_0$: $$ g_0(y)= \big[ F_1 (\langle y, \nabla \psi^*(y)\rangle -\psi^*(y)) \cdot F_2(\psi^*(y))^{\frac{p}{n}}  \cdot  \det \,  \nabla^2  \psi^*(y) \big]^{\frac{n}{n+p}}.$$  
  
  \subsection{The general Orlicz affine and geominimal surface areas for convex functions} \label{section 3.2}  

  Let $h: (0, \infty)\rightarrow (0, \infty)$ be a continuous function and $\psi\in \C$.    \begin{defn} \label{Orlicz-mixed-integral-log-concave-1} For  measurable functions $F_1, F_2 : \R\rightarrow (0, \infty)$ and  $g\in \F$,  define  the Orlicz  mixed integral of $\psi$ and $g$ with respect to $F_1$ and $F_2$ by  \begin{eqnarray*}V_{h, F_1, F_2} (\psi, g)= \int_{X_\psi} h \left(\frac{g(\nabla \psi(x))}{F_2( \langle x, \nabla \psi(x)\rangle - \psi(x)) }\right)  F_1(\psi(x))\,dx.\end{eqnarray*}   \end{defn} When $h(t)=t^{-p/n}$, one recovers formula (\ref{mixed:lp-1}).   Moreover, if $g=\tau \cdot (F_2\circ \psi^*)$ for some constant $\tau>0$,  \begin{equation}\label{special mixed integral-1} V_{h, F_1, F_2} \big( \psi,  \tau \cdot (F_2\circ \psi^*)\big)=h(\tau)\cdot I(F_1\circ \psi, \psi). \end{equation}

Denote by $GL(n)$ the set of all invertible linear maps on $\bbR^n$. For $T\in GL(n)$,  we use $\det (T)$  or $\det T$  for  the determinant of $T$.  Let $SL_{\pm}(n)$ denote the subset of $GL(n)$ which contains all $T\in GL(n)$ such that $\det (T)=\pm 1$. The inverse of $T$ is written by $T^{-1}$ and the transpose of $T$ is written as $T^t$. For convenience, the inverse of $T^t$ is denoted by $T^{-t}$.  

For $T\in SL_{\pm}(n)$ and $g\in \F$, by formula (\ref{Legendre}), one has, 
    \begin{eqnarray*} (\psi\circ T)^* ( y ) &=& \sup_{x\in \bbR^n}  \bigl( \sca{x,y} - (\psi\circ T) (x) \bigr)=  \sup_{x\in \bbR^n}  \bigl( \sca{Tx,T^{-t}y} - \psi(Tx)  \bigr) \\&=& \psi^*(T^{-t}y)=\big(\psi^*\circ T^{-t}\big)(y).
\end{eqnarray*}    Hence, $y\in X_{(\psi\circ T)^*}$  if and only if $y\in T^t (X_{\psi^*})$, which follows from the general fact $$\nabla^2  (\psi\circ T)(x) =T^t \big(\nabla^2  \psi(Tx)\big) T$$ for $T\in GL(n)$.  This implies $g\circ T^{-t}\in \mathcal{F}^+_{(\psi\circ T)^*}$ if $g\in \F$. Moreover, by $|\det(T^{-t})|=1$ and by  $y=T^{-t} z$, formula (\ref{ratio of g}) implies  \begin{eqnarray}  I(g, \psi^*) &=&  \int_{X_{\psi^*}} g(y) \, dy=  \int_{T^{t}(X_{\psi^*})} g(T^{-t}z) \, dz \nonumber \\ &
     =&  \int_{ X_{(\psi\circ T)^*}} (g\circ T^{-t})(z) \, dz=I\big(g\circ T^{-t},  (\psi\circ T)^*\big). \label{I-g-invariant}
    \end{eqnarray} Moreover, we can prove that  the above defined  Orlicz mixed integral is $SL_{\pm}(n)$-invariant.
\begin{lem}\label{affine-invariance:mixed-1} Let $F_1, F_2, \psi$ and $g$ be as in Definition \ref{Orlicz-mixed-integral-log-concave-1}. Then, for all $T\in SL_{\pm}(n)$, one has, $$V_{h, F_1, F_2} (\psi\circ T, g\circ T^{-t})=V_{h, F_1, F_2} (\psi, g).$$ 
\end{lem}  \begin{proof}   Let $T\in SL_{\pm}(n)$.  Recall that $\nabla (\psi\circ T)(x)=T^t \nabla \psi(Tx)$, which implies  $x\in X_{\psi\circ T}$ if and only if $Tx\in X_{\psi}$. Hence, by letting $y=Tx$, one has, \begin{eqnarray*}V_{h, F_1, F_2} (\psi\circ T, g\circ T^{-t})&=&  \int_{X_{\psi\circ T}} h\bigg(\frac{(g\circ T^{-t})\big(\nabla (\psi\circ T)(x)\big)}{F_2( \langle x, \nabla (\psi\circ T)(x)\rangle - (\psi\circ T)(x))}\bigg)  F_1\big((\psi\circ T)(x)\big)\,dx\\ &=& \int_{X_{\psi\circ T}} h\left(\frac{g(\nabla \psi(Tx))}{F_2( \langle x,T^t \nabla \psi(Tx)\rangle -  \psi(Tx))}\right)  F_1(\psi(Tx))\,dx\\  &=& \int_{X_{\psi}} h\left(\frac{g(\nabla \psi(y))}{F_2( \langle y, \nabla \psi(y)\rangle -  \psi(y))}\right)  F_1(\psi(y))\,dy\\&=& V_{h, F_1, F_2} (\psi, g).\end{eqnarray*}     \end{proof}

The following function classes were defined in \cite{Ye2014a2} and will play fundamental roles in this paper. Let \begin{eqnarray*} \Phi &=&\{h:   \mbox{$h$ is either a constant or a strictly convex function}\}; \\   \Psi&=&\{h: h\ \mbox{is either a constant or an increasing strictly concave function}\}.\end{eqnarray*}  Throughout this paper,  $\M$ refers to the subset of $\F$ which contains all log-concave functions. Note that log-concave functions are analogous to convex bodies in geometry; and hence $\M$ is used to define the general  Orlicz geominimal surface area of convex functions (although $\psi$ or $F_1\circ \psi$ or $F_2\circ\psi^*$ may not be log-concave).  Motivated by Theorem \ref{equivalent:affine:surface:area}, the general Orlicz affine and geominimal surface areas of  $\psi$  could be defined as follows.  
 
  \begin{defn} \label{Orlicz affine surface}   For $h\in \Phi$,   the general  Orlicz affine surface area of  $\psi\in \C$ is defined by \begin{eqnarray*}
  \OrliczA(\psi)=\inf_{g\in \F} \bigg\{V_{h, F_1, F_2} \bigg(\psi, \frac{(\sqrt{2\pi})^n\cdot g}{I(g, \psi^*)}\bigg)\bigg\}; 
  \end{eqnarray*} and the general  Orlicz geominimal surface area of  $\psi\in \C$ is defined by \begin{eqnarray*}
  \OrliczG(\psi)=\inf_{g\in \M} \bigg\{V_{h, F_1, F_2} \bigg(\psi, \frac{(\sqrt{2\pi})^n\cdot g}{I(g, \psi^*)}\bigg)\bigg\}. 
  \end{eqnarray*} 
   When $h\in \Psi$, $\OrliczA(\psi)$ and $ \OrliczG(\psi)$  are defined as above but with `` $\inf$" replaced by `` $\sup$". 
\end{defn} 

\noindent {\bf Remark.} The above definitions could be extended to more general cases with $\F$ and $\M$ replaced by any subset of $\F$; and the properties would be similar to those for $ \OrliczA(\psi)$ and $ \OrliczG(\psi)$ which are the most important cases.  In fact, one can let $I(g, \psi^*)=(\sqrt{2\pi})^n$ in Definition \ref{Orlicz affine surface}, for instance, if $h\in \Phi$, 
$$\OrliczA(\psi)=\inf \left\{ V_{h, F_1, F_2} (\psi, g): g\in \F \ \mathrm{with}\ I(g, \psi^*)=(\sqrt{2\pi})^n \right\}.$$   It can be also easily checked that  $\OrliczA(\psi) \leq \OrliczG(\psi)$  for $h\in \Phi$ and $\OrliczA(\psi) \geq  \OrliczG(\psi)$ for $h\in \Psi$.  Moreover,  $\OrliczA(\psi) =\OrliczG(\psi) =  I(F_1\circ \psi, \psi)$ if $h(t)=1$.

 If $F_1(t)=F_2(t)=e^{-t}$ and  $\psi$ is  a convex function, then  $f=F_1\circ \psi=e^{-\psi}$ and $F_2\circ\psi^*=e^{-\psi^*}=f^{\circ}$ (the polar dual of $f$) are log-concave functions. Therefore, one can define the Orlicz affine surface area of the log-concave function $f=e^{-\psi}$ by $as_{h}^{orlicz}(f)=as_{h, e^{-t}, e^{-t}}^{orlicz}(\psi).$  This serves as a non-homogeneous extension of the $L_p$ affine surface area of log-concave functions \cite{Caglar-6, CaglarWerner}.  Similarly,  $G_{h}^{orlicz}(f)=G_{h, e^{-t}, e^{-t}}^{orlicz}(\psi)$ defines the Orlicz geominimal surface area of $f$, which is new to the literature.

 The following theorem states that the general  Orlicz affine and geominimal surface areas of  $\psi$ are $SL_{\pm}(n)$-invariant.  \begin{theo} \label{Affine-invariance--1} Let $\psi\in \C$. For $T\in SL_{\pm}(n)$ and $h\in \Phi\cup\Psi$, one has $$\OrliczA(\psi\circ T)=\OrliczA(\psi) \ \ \mathrm{and} \ \ \OrliczG(\psi\circ T)=\OrliczG(\psi).$$  In particular, $as_{h}^{orlicz}(f)$ and  $G_{h}^{orlicz}(f)$ are $SL_{\pm}(n)$-invariant.    \end{theo}  
 
 \begin{proof}  We only prove the case for   $\OrliczA(\psi)$ and the case for $\OrliczG(\psi)$  follows along the same lines. The desired result follows from Lemma \ref{affine-invariance:mixed-1}, formula (\ref{I-g-invariant}) and the remark after Definition \ref{Orlicz affine surface}: for $h\in \Phi$,  
\begin{eqnarray*} \OrliczA(\psi) \!\!&=&\!\!\! \inf \left\{ V_{h, F_1, F_2}(\psi, g): g\in \F \ \mathrm{with} \ I(g, \psi^*)=(\sqrt{2\pi})^n \right\} \\ \!\!&=&\!\!\! \inf \left\{ V_{h, F_1, F_2}(\psi\circ T, g\circ T^{-t}): g\in \F \ \mathrm{with} \ I(g, \psi^*)=(\sqrt{2\pi})^n  \right\} \\\!\! &=&\!\!\! \inf \big\{ V_{h, F_1, F_2}(\psi\circ T, g\circ T^{-t}): g\circ T^{-t}\in \mathcal{F}^+_{(\psi\circ T)^*}  \ \mathrm{with} \ I\big(g\circ T^{-t}, (\psi\circ T)^*\big)=(\sqrt{2\pi})^n  \big\} \\\!\!&=&\!\!\! \OrliczA\big(\psi\circ T\big). \end{eqnarray*}  Replacing ``$\inf$" by ``$\sup$", one gets  the $SL_{\pm}(n)$-invariance of $\OrliczA(\psi)$ for $h\in \Psi$.  
\end{proof}

Let $c>0$ be a constant and $F:\bbR\rightarrow (0, \infty)$ be a measurable function. For convenience, let    \begin{eqnarray*}    I(F, c)= \int_{\bbR^n} F\Big(\frac{c^2\|x\| ^2}{2}\Big)\,dx.\end{eqnarray*}   It can be checked that 
 \begin{eqnarray}\label{variable-change-1} I(F, c)=c^{-n}\cdot  I(F, 1). \end{eqnarray} 
 
 The following corollary provides the precise values of $as_{h, aF, bF}^{orlicz}\big(\frac{c^2\|\cdot\| ^2}{2} \big)$ and $G_{h, aF, bF}^{orlicz}\big(\frac{c^2\|\cdot\| ^2}{2} \big)$ with constants $a, b>0$.  When $a=b=1$ and $F(t)=e^{-t}$, one gets $$as^{orlicz}_{h}\big(\gamma_n\circ c\big)= G^{orlicz}_{h}\big(\gamma_n\circ c\big) =  c^{-n} \cdot h (c^{-n})\cdot (\sqrt{2\pi})^n,$$ where $(\gamma_n\circ c)(x)=\gamma_n(cx)$ for $x\in \bbR^n$ and $\gamma_n(x)=e^{-\frac{\|x\|^2}{2}}$.  Note that $\gamma_n^\circ=\gamma_n$, and hence $\gamma_n$ serves as the ``Euclidean unit ball" in the geometry of log-concave functions. 
  
  
 \begin{cor} \label{equality of ball-1}  Let $a, b, c>0$ be constants and  $F \colon \R\to (0, \infty)$ be a measurable function such that  $0< I\big(F, 1 \big)<\infty$. 
Then, for $h\in \Phi\cup\Psi$, $$ as_{h, aF, bF}^{orlicz}\Big(\frac{c^2\|\cdot\| ^2}{2} \Big) = a \cdot I\big(F, c \big)\cdot h \Big(\frac{(\sqrt{2\pi})^n}{c^{2n}\cdot b \cdot   I(F, c)}\Big) .$$  The same formula holds for $G_{h, aF, bF}^{orlicz}\big(\frac{c^2\|\cdot\| ^2}{2} \big)$ if the function $F\big(\frac{\|\cdot\| ^2}{2}\big)$ is log-concave.  \end{cor}   \begin{proof} We only prove the case $h\in\Phi$ and the proof for the case $h\in\Psi$ follows along the same line. Note that $X_{\frac{c^2\|\cdot\| ^2}{2}}=\bbR^n$,  $\nabla \frac{c^2\|x\| ^2}{2}=c^2 x$ and $\langle x, \nabla\frac{c^2\|x\| ^2}{2}\rangle -\frac{c^2\|x\| ^2}{2}= \frac{c^2 \|x\| ^2}{2}$. Applying Jensen's inequality to the convex function $h$, one has, for all $g: \bbR^n\rightarrow (0, \infty)$ with $\int_{\bbR^n} g(y)\,dy>0,$   \begin{eqnarray*}V_{h, aF, bF} \Big(\frac{c^2\|\cdot\| ^2}{2}, g\Big)  &=& a\cdot I\big(F, c\big)\cdot \int_{\R^n} h\bigg(\frac{g(c^2x)}{b\cdot F\big(\frac{c^2\|x\| ^2}{2}\big)}\bigg)  \frac{F\big(\frac{c^2\|x\| ^2}{2}\big)}{I\big(F, c \big)}\,dx\\ &\geq& a\cdot I\big(F, c \big) \cdot  h\bigg(\int_{\R^n} \frac{g(c^2x)}{b\cdot I\big(F, c \big)}\,dx\bigg)\\ &=&a\cdot I\big(F, c \big) \cdot  h\bigg( \frac{1}{c^{2n} \cdot b\cdot I\big(F, c \big)} \int _{\bbR^n} g(y)\,dy\bigg).\end{eqnarray*} This leads to, for $h\in \Phi$, \begin{eqnarray}
 as_{h, aF, bF}^{orlicz}\Big(\!\frac{c^2 \|\cdot\| ^2\!}{2} \Big)\!\!\!\!&=&\!\!\! \inf \left\{\!V_{h, aF, bF} \Big(\!\frac{c^2\|\cdot\| ^2}{2}, g\!\Big)\!\!:  g \ \mathrm{is\ a\ positive\ function\ on} \ \bbR^n\ \mathrm{and}\  \int_{\bbR^n}\! g(y)\,dy=(\sqrt{2\pi})^n\! \right\} \nonumber \\ 
 \!\!\!\! &\geq&\!\!\! a \cdot  I\big(F, c \big) \cdot  h\bigg(\frac{(\sqrt{2\pi})^n}{c^{2n}\cdot b \cdot I\big(F, c \big)}  \bigg).\nonumber \end{eqnarray} 
 On the other hand,  by formulas (\ref{special mixed integral-1}) and (\ref{variable-change-1}),  and Definitions \ref{Orlicz-mixed-integral-log-concave-1} and \ref{Orlicz affine surface}, one can check \begin{eqnarray}
  as_{h, aF, bF}^{orlicz}\Big(\frac{c^2 \|\cdot\| ^2}{2} \Big) \leq   V_{h, aF, bF}\Big(\frac{c^2 \|\cdot\| ^2}{2} , \frac{ (\sqrt{2\pi})^n \cdot F\big(\frac{\|\cdot \| ^2}{2c^{2}}\big)}{I(F, c^{-1})}\Big)  = a\cdot I\big(F, c \big) \cdot  h\bigg( \frac{(\sqrt{2\pi})^n}{c^{2n}\cdot b \cdot I\big(F, c \big)}  \bigg),\label{eq-eq-2} \end{eqnarray}  and the desired result follows. 
  
 The proof for  $G_{h, aF, bF}^{orlicz}\big(\frac{c^2\|\cdot\| ^2}{2} \big)$ follows along the same lines. The additional assumption that $F\big(\frac{\|\cdot\| ^2}{2}\big)$  is log-concave  is needed to obtain inequality (\ref{eq-eq-2}).        \end{proof}

\subsection{Inequalities}\label{section 3.3}  

   In this subsection, we prove some inequalities for the general Orlicz affine and geominimal surface areas of convex functions.  Hereafter, we always assume that $$I(F_1\circ\psi, \psi)\in (0, \infty) \ \ \ \mathrm{and} \ \ \  I(F_2\circ\psi^*, \psi^*)\in (0, \infty).$$ In particular, when $F_1(t)=F_2(t)=e^{-t}$,   we assume that  $$I(f)=I(e^{-t}\circ \psi, \psi)\in (0, \infty) \ \ \mathrm{and} \ \ I(f^\circ)=I(e^{-t}\circ \psi^*, \psi^*)\in (0, \infty),$$ where $f=e^{- \psi}$ and $f^\circ=e^{-\psi^*}$ are log-concave functions.

The following proposition is needed in order to prove some inequalities for the general Orlicz affine and geominimal surface areas of convex functions.   

 \begin{prop}\label{bounded by volume product} Let $\psi\in \C$. Then, for $h\in \Phi$, \begin{eqnarray*}   \OrliczA(\psi)    \leq I(F_1\circ \psi, \psi) \cdot h\Big(\frac{(\sqrt{2\pi})^n}{I(F_2\circ\psi^*, \psi^*)}\Big), \end{eqnarray*} and if in addition  $F_2\circ \psi^*$ is log-concave,  \begin{eqnarray*}   \OrliczA(\psi)    \leq  \OrliczG(\psi)    \leq I(F_1\circ \psi, \psi) \cdot h\Big(\frac{(\sqrt{2\pi})^n}{I(F_2\circ\psi^*, \psi^*)}\Big).  \end{eqnarray*}  In particular,  for $h\in \Phi$ and $f=e^{- \psi}$, \begin{eqnarray*} as_{h}^{orlicz} (f)\leq G_{h}^{orlicz} (f)  \leq I(f) \cdot h\bigg(\frac{(\sqrt{2\pi})^n}{I(f^\circ)}\bigg).\end{eqnarray*}  
The above inequalities hold  for $h\in \Psi$  with ``$\leq$" replaced by ``$\geq$".  \end{prop}  
 
 \begin{proof} Formula (\ref{special mixed integral-1}) and Definition \ref{Orlicz affine surface} imply that for $h\in \Phi$,  \begin{eqnarray*}   \OrliczA(\psi)     \leq   V_{h, F_1, F_2}\bigg(\psi, \frac{(\sqrt{2\pi})^n\cdot (F_2\circ \psi^*)}{I(F_2\circ \psi^*, \psi^*)}\bigg)  = I(F_1\circ \psi, \psi) \cdot h\Big(\frac{(\sqrt{2\pi})^n}{I(F_2\circ\psi^*, \psi^*)}\Big);\end{eqnarray*}  while for $h\in \Psi$, similar inequality holds with ``$\leq$" replaced by ``$\geq$".  
 
 The desired result for $\OrliczG(\psi)$ follows along the same lines  if in addition $F_2\circ \psi^*\in \M$.  \end{proof} 

For measurable functions $F_1, F_2: \bbR\rightarrow (0, \infty)$, define the  decreasing function  $\breve{F}:\R \to (0, \infty)$ by 
\begin{equation*}
\breve{F}(t)= \sup_{\frac{t_1+ t_2}{2} \ge t} \sqrt{F_1(t_1) F_2 ( t_2 )}.   
\end{equation*}  It can be checked that  $\breve{F}=F_1=F_2$ if $F_1=F_2$ is a log-concave and  decreasing function.  Let $$I(\breve{F}, c)=\int_{\bbR^n} \breve{F}\Big(\frac{c^2\|x\| ^2}{2}\Big)\,dx.$$  

 For $z\in \bbR^n$ and  for $\psi\in \C$, let $\psi_z(x)=\psi(x+z)$ and $\psi_z^*=(\psi_z)^*$.  It was proved in \cite{Caglar-6} (as a direct consequence of the functional Blaschke-Santal\'o inequality  \cite{FradeliziMeyer2007, Lehec2009b}) that  there exists $z_0\in\R^n$ such that 
\begin{equation*}
I(F_1 \circ \psi, \psi) \cdot I(F_2 \circ \psi_{z_0}^*, \psi_{z_0}^*)
\leq \big[I\big(\breve{F}, 1\big) \big]^2. \end{equation*}  
 Let $\C_0$ be the set of convex functions in $\C$ with $z_0=0$.  Therefore, for all $\psi\in \C_0$, one has  \begin{equation}
\label{eq:BSfunctional}
I(F_1 \circ \psi, \psi) \cdot I(F_2 \circ \psi^*, \psi^*)
\leq \big[I\big(\breve{F}, 1\big) \big]^2.
\end{equation}  If in addition $\breve{F}$ is strictly decreasing and $I(F_1 \circ \psi, \psi) \ne 0$ (or $I(F_2 \circ \psi^*, \psi^*)\neq 0$), equality holds in inequality (\ref{eq:BSfunctional}) if and only if there exist $ b \in (0, \infty)$, $a \in \R$ and a positive definite matrix $A$ such that for every $x\in\R^n$ and $t\ge 0$, 
\begin{equation}
\psi(x)=\langle Ax,x\rangle+a,\quad F_1(t+a)=b\breve{F}(t) \quad{\rm and}\quad bF_2(t-a)= \breve{F}(t). \label{equality:character}
\end{equation}
In particular, for log-concave function $f=e^{-\psi}$, inequality (\ref{eq:BSfunctional}) becomes the classical functional Blaschke-Santal\'{o} inequality  \cite{ArtKlarMil, KBallthesis, FradeliziMeyer2007,Lehec2009}:    
\begin{equation*} 
I(f) \cdot I(f^\circ)
\leq  (2\pi)^n,
\end{equation*} with equality  if and only if there exist $a \in \R$ and a positive definite matrix $A$ such that \begin{equation}
\psi (x)=\langle Ax,x\rangle+a, \ \ \ \ \mathrm{for}\ x\in\R^n. \label{equality:character-logconcave}
\end{equation}   

Now we can prove the following functional affine isoperimetric inequalities, which provide upper bound (lower bound, respectively)  for the general  Orlicz affine and geominimal surface areas for $h\in \Phi$ (for $h\in \Psi$ respectively). For convenience, let  $$\hat{c}=\bigg(\frac{I(\breve{F}, 1)}{I(F_2\circ \psi^*, \psi^*)}\bigg)^{\frac{1}{n}}  \ \ \mathrm{and}  \ \ \bar{c}=\bigg(\frac{I(\breve{F}, 1)}{I(F_1\circ\psi, \psi)}\bigg)^{\frac{1}{n}}.$$

\begin{theo}\label{generalaffineineq}  Let $F_1, F_2 \colon \R \to(0, \infty)$ be measurable functions such that $0<I( \breve{F}, 1)<\infty$. Let $\psi\in \C_0$.  

 \noindent  (i)  For $h\in \Phi$,  one has, 
  \[ \OrliczA(\psi)  
\le as^{orlicz}_{h, \breve{F}, \breve{F}} \Big(\frac{\|\cdot \|^2}{2\cdot \hat{c}^2}\Big),\] and if in addition both $F_2\circ \psi^*$ and $\breve{F}(\frac{\|\cdot \| ^2}{2})$ are log-concave,   \[ \OrliczG(\psi)  
\le G^{orlicz}_{h, \breve{F}, \breve{F}} \Big(\frac{\|\cdot \|^2}{2\cdot \hat{c}^2}\Big).\] 
  If in addition $\breve{F}$  is strictly decreasing,  equality holds if and only if  $F_1, F_2, \breve{F}, \psi$ satisfy formula (\ref{equality:character}).
\vskip 2mm  \noindent 
 (ii) For $h\in \Phi$ being a decreasing function,   one has,  \[ \OrliczA(\psi) 
\le as^{orlicz}_{h, \breve{F}, \breve{F}} \Big(\frac{ \bar{c}^2\cdot \|\cdot \|^2}{2 }\Big),\] and if in addition both $F_2\circ \psi^*$ and $\breve{F}(\frac{\|\cdot \| ^2}{2})$ are log-concave,  
\[ \OrliczG(\psi) 
\le G^{orlicz}_{h, \breve{F}, \breve{F}} \Big(\frac{ \bar{c}^2\cdot \| \cdot \|^2}{2 }\Big). \]
  The above inequalities hold for $h\in \Psi$ with `` $\leq$" replaced by `` $\geq$". 

 Moreover, if $h\in \Phi$ is strictly decreasing (or $h\in \Psi$ is strictly increasing) and $\breve{F}$ is strictly decreasing,  equality holds if and only if $F_1, F_2, \breve{F}, \psi$ satisfy formula (\ref{equality:character}). \end{theo}

\begin{proof} (i).  First, $I(F_2\circ \psi^*, \psi^*)=\hat{c}^{-n} I (\breve{F}, 1)=I(\breve{F}, {\hat{c}})$ by formula (\ref{variable-change-1}). Inequality (\ref{eq:BSfunctional}) implies  $$I(F_1\circ\psi, \psi)\leq  \hat{c}^n\cdot I(\breve{F}, 1)=I(\breve{F}, \hat{c}^{-1}).$$
 Proposition \ref{bounded by volume product} implies that  for all $h\in \Phi$, \begin{eqnarray}   \OrliczA(\psi)   \leq I(F_1\circ \psi, \psi) \cdot h\Big(\frac{(\sqrt{2\pi})^n}{I(F_2\circ\psi^*, \psi^*)}\Big) \leq I(\breve{F},  \hat{c}^{-1}) \cdot h\Big(\frac{ (\sqrt{2\pi})^n}{\hat{c} ^{-2n} \cdot I(\breve{F}, \hat{c}^{-1})}\Big),\label{1-3-4-5}\end{eqnarray}  and hence  the desired result follows from Corollary \ref{equality of ball-1}.  

Now let us characterize the condition for equality. First, assume that $F_1, F_2, \breve{F}, \psi$ satisfy formula (\ref{equality:character}). Letting $A=T^t T$ and $z=\sqrt{2} Ty$, one has, 
\begin{eqnarray*} I\big(\breve{F} (\langle Ax, x\rangle)\big) =   \int_{\bbR^n} \breve{F}\Big(\frac{\|\sqrt{2} Tx\|^2}{2}\Big) \,dx =  \frac{1}{\sqrt{2^n\cdot \det A}} \int_{\bbR^n} \breve{F} \Big(\frac{\|z\|^2}{2}\Big) \,dz= \frac{I(\breve{F}, 1)}{\sqrt{2^n\cdot \det A}}. \end{eqnarray*} Similar to the proof of Corollary \ref{equality of ball-1}, one can show that 
\begin{eqnarray}   \OrliczA(\psi)&=& b \cdot I\big(\breve{F} (\langle Ax, x\rangle)\big)\cdot h\bigg(\frac{b\cdot (\sqrt{2\pi})^n}{2^n \cdot \det A\cdot I\big(\breve{F}(\langle Ax, x\rangle)\big)}\bigg)\nonumber \\ &=& \frac{b \cdot I(\breve{F}, 1)}{\sqrt{2^n\cdot \det A}}  \cdot h\bigg(\frac{b\cdot (\sqrt{2\pi})^n}{\sqrt{2^n \cdot \det A} \cdot I\big(\breve{F}, 1\big)}\bigg). \label{233-333}
\end{eqnarray}  For $\psi(x)=\langle Ax, x\rangle +a$, one has $\psi^*(y)=\frac{1}{4}\langle A^{-1} y, y\rangle-a$ and hence, \begin{eqnarray*} I(F_2\circ \psi^*, \psi^*) &=& \frac{1}{b} \int_{\bbR^n} \breve{F}\Big(\frac{\langle A^{-1}y, y\rangle}{4}\Big) \,dy  = \frac{\sqrt{2^n\cdot \det A}}{b}\  I(\breve{F}, 1),\\  \hat{c}^n&=&\frac{I(\breve{F}, 1)}{I(F_2\circ \psi^*, \psi^*)} =\frac{b}{\sqrt{2^n\cdot \det A}}. \end{eqnarray*}  Corollary \ref{equality of ball-1}, formula  (\ref{variable-change-1}) and formula (\ref{233-333}) imply that if  $F_1, F_2, \breve{F}, \psi$ satisfy formula (\ref{equality:character}),   \begin{eqnarray*}  as^{orlicz}_{h, \breve{F}, \breve{F}} \Big(\frac{\|\cdot \|^2}{2\cdot \hat{c}^2}\Big) &=& I(\breve{F},  \hat{c}^{-1}) \cdot h\Big(\frac{ (\sqrt{2\pi})^n}{\hat{c} ^{-2n} \cdot I(\breve{F}, \hat{c}^{-1})}\Big)\\ &=& \hat{c}^{n} I(\breve{F},  1) \cdot h\Big(\frac{ (\sqrt{2\pi})^n}{\hat{c} ^{-n} \cdot I(\breve{F}, 1)}\Big) \\ &=& \frac{b\cdot I(\breve{F},  1)}{\sqrt{2^n\cdot \det A}}  \cdot h\Big(\frac{b\cdot (\sqrt{2\pi})^n}{ \sqrt{2^n\cdot \det A} \cdot I(\breve{F}, 1)}\Big) \\ &=&   \OrliczA(\psi). \end{eqnarray*}   

On the other hand, if $\breve{F}$ is strictly decreasing, then equality holds in (\ref{1-3-4-5}) only if equality holds in inequality (\ref{eq:BSfunctional}). That is,  $F_1, F_2, \breve{F}, \psi$ satisfy formula (\ref{equality:character}).  Hence, we have verified the desired characterization of equality in (i).  
 
 \vskip 2mm \noindent (ii). By inequality (\ref{eq:BSfunctional}), one can check that $I(F_2\circ \psi^*, \psi^*)\leq \bar{c}^nI(\breve{F}, 1).$ Proposition \ref{bounded by volume product} implies that for all decreasing $h\in \Phi$,  \begin{eqnarray*}  \OrliczA(\psi)  \leq I(F_1\circ \psi, \psi) \cdot h\Big(\frac{(\sqrt{2\pi})^n}{I(F_2\circ\psi^*, \psi^*)}\Big) \leq I(\breve{F}, {\bar{c}})  \cdot  h\Big(\frac{(\sqrt{2\pi})^n}{\bar{c}^{2n} \cdot I(\breve{F}, {\bar{c}})}\Big),\end{eqnarray*}  and hence  the desired result follows from Corollary \ref{equality of ball-1}.   The characterization of equality follows along the same lines as in (i).

The desired results for  $ \OrliczG(\psi)$ follow along the same lines. The additional assumptions that both $F_2\circ \psi^*$ and $\breve{F}(\frac{\|\cdot \| ^2}{2})$ are log-concave are needed in order to use   Proposition \ref{bounded by volume product} and Corollary \ref{equality of ball-1}.    \end{proof}  
 
 The following result follows immediately from Theorem \ref{generalaffineineq} by letting $F_1(t)=F_2(t)=e^{-t}$. These affine isoperimetric inequalities state that the maximum (minimum, respectively) of $as^{orlicz}_{h}(f)$ and $G^{orlicz}_{h}(f)$ for $h\in \Phi$ (for $h\in\Psi$, respectively) attain at (and only at) the Gaussian functions.

\begin{cor}\label{generalaffineineq-2}  Let  $\psi\in \C_0$ and $f=e^{-\psi}$.  

  \noindent  (i)  For $h\in \Phi$, one has,  \[ as^{orlicz}_{h}(f) \leq G^{orlicz}_{h}(f) 
\le G^{orlicz}_{h}\Big(\!\exp\!\Big(\!-\frac{ \|\cdot \|^2}{4\pi}\cdot \big[I\big(f^\circ\big)\big]^{\frac{2}{n}}\Big)\Big).\]   Equality holds if and only if  $\psi$  satisfies formula (\ref{equality:character-logconcave}).
\vskip 2mm  \noindent 
 (ii) For decreasing $h\in \Phi$, one has,
 \[ as^{orlicz}_{h}(f) \leq G^{orlicz}_{h}(f) 
\le G^{orlicz}_{h} \Big(\!\exp\Big(\!-\pi  \|\cdot \|^2 \cdot \big[I\big(f \big)\big]^{-\frac{2}{n}}\Big)\Big).\] The above inequality holds for $h\in \Psi$ with `` $\leq$" replaced by `` $\geq$".

   If $h\in \Phi$ is strictly decreasing (or  $h\in \Psi$ is strictly increasing),  equality holds if and only if $\psi$ satisfies formula (\ref{equality:character-logconcave}). \end{cor}

 We now establish cyclic inequalities for the general  Orlicz affine and geominimal surface areas of convex functions. Assume that the function $h_1$ always has inverse $h_1^{-1}$ and let $H=h\circ h_1^{-1}$.  Moreover, $H(0)$ and $H(\infty)$ are defined by the limit of $H(t)$ as $t\rightarrow 0$ and $t\rightarrow \infty$ respectively (could be a finite number or $\infty$, if exist).  Note that if  $h^{-1}$ and $h^{-1}_1$ both exist, then condition (a) is equivalent to condition (d); and condition (c) is equivalent to condition (f) if in addition $H$ is increasing.    
\begin{theo}\label{cyclic}  Let  $F_1, F_2 \colon \R\to (0, \infty)$ be measurable functions and $\psi\in \C$.  

 \noindent
(i) Assume one of the following conditions: (a) $h\in \Phi$ and $h_1\in \Psi$ with $H$ increasing; (b) $h, h_1\in \Phi$ with $H$ decreasing; (c) $H$ concave increasing with either $h, h_1 \in \Phi$ or $h, h_1 \in \Psi$. Then \begin{equation*}
\frac{\OrliczA(\psi)}{I(F_1\circ \psi, \psi)} \leq  H\bigg(\frac{as^{orlicz}_{h_1, F_1, F_2}(\psi)}{I(F_1\circ \psi, \psi)} \bigg) . \end{equation*}
 (ii) Assume one of the following conditions: (d) $h\in \Psi$ and $h_1\in \Phi$ with $H$ increasing; (e) $H$ convex decreasing with one in $\Phi$ and the other one in $\Psi$; (f) $H$  convex increasing with either $h, h_1\in \Phi$ or $h, h_2\in \Psi$. Then \begin{equation*}
\frac{\OrliczA(\psi)}{I(F_1\circ \psi, \psi)} \geq  H\bigg(\frac{as^{orlicz}_{h_1, F_1, F_2}(\psi)}{I(F_1\circ \psi, \psi)} \bigg) . \end{equation*}  

The same inequalities also hold for the general  Orlicz geominimal surface area of convex functions, if in addition $F_2\circ \psi^*\in \M$ in conditions (a), (b) and (d).   \end{theo}

  \noindent {\bf Remark.} In particular, the above inequalities hold for $as^{orlicz}_{h}(f)$ and $G^{orlicz}_{h}(f)$, as long as corresponding conditions are verified.  

 \begin{proof}  For completeness, we include a brief proof, which is similar to that of Theorem 3.1 in \cite{Ye2014a2}. Results for conditions (a), (b) and (d) follow immediately from  Proposition \ref{bounded by volume product} and the monotonicity of $H$.    Results for conditions (c), (e) and (f) hold by the combination of Jensen's inequality, the monotonicity of $H$, and Definition \ref{Orlicz affine surface}. Here, as an example, we show the case for condition (c) and omit the proofs for other cases.   Jensen's inequality to the concave function $H$ implies  \begin{eqnarray}\frac{V_{h, F_1, F_2}(\psi, g)}{I(F_1\circ \psi, \psi)}   \leq  H\bigg(\frac{V_{h_1, F_1, F_2} (\psi, g)}{I(F_1\circ \psi, \psi)}\bigg)\nonumber.\end{eqnarray} As $H$ is increasing and $h,h_1\in \Phi$, then \begin{eqnarray}\frac{\OrliczA(\psi)}{I(F_1\circ\psi, \psi)} &=&\inf\bigg\{\frac{ V_{h, F_1, F_2}\big(\psi, g \big)}{I(F_1\circ \psi, \psi)}:   g\in \F \ \mathrm{with} \  I(g, \psi^*)=(\sqrt{2\pi})^n  \bigg\} \nonumber \\ &\leq& H\bigg( \inf\Big\{ \frac{ V_{h_1, F_1, F_2}(\psi, g)}{I(F_1\circ \psi, \psi)}:  g\in \F \ \mathrm{with} \  I(g, \psi^*)=(\sqrt{2\pi})^n  \Big\}\bigg)\nonumber\\ &=&  H\bigg(\frac{as^{orlicz}_{h_1, F_1, F_2} (\psi)}{I(F_1\circ\psi, \psi)}\bigg)\nonumber.\end{eqnarray} The case $h, h_1\in \Psi$ follows similarly with `` $\inf$ " replaced by `` $\sup$". \end{proof}

 \subsection{The general  $L_p$ geominimal surface area for convex functions and a Santal\'{o} type inequality}
Theorem \ref{equivalent:affine:surface:area} and Definition \ref{Orlicz affine surface} yield that $$as_{p, F_1, F_2}(\psi) =(\sqrt{2\pi})^{\frac{np}{n+p}}\big[\OrliczA(\psi)\big]^{\frac{n}{n+p}} $$ with $h(t)=t^{-p/n}$ for $-n\neq p\in \bbR$. Its properties have been discussed in e.g.  \cite{Caglar-6, CaglarWerner}. 

 In this subsection, we briefly discuss properties for the general $L_p$ geominimal surface areas of convex functions  for $-n\neq p\in \bbR$.  Taking Theorem \ref{equivalent:affine:surface:area} into account, it is more natural to define the general $L_p$ geominimal surface areas of convex functions as $$G_{p, F_1, F_2}(\psi) =(\sqrt{2\pi})^{\frac{np}{n+p}}\big[\OrliczG(\psi)\big]^{\frac{n}{n+p}} $$ with $h(t)=t^{-p/n}$ for $-n\neq p\in \bbR$.

  \begin{defn}
  \label{equivalent:affine:surface:area-homogeneous p case}  For $p\geq 0$,  define  the general  $L_p$ geominimal surface area of $\psi\in \C$ by   \begin{equation*}G_{p, F_1, F_2}(\psi)=\inf_{ g\in \M } \left\{ V_{p, F_1, F_2} (\psi, g)^{\frac{n}{n+p}}\
I(g, \psi^*)^{\frac{p}{n+p}}\right\}; \end{equation*}  while for $-n\neq p<0$, $G_{p, F_1, F_2}(\psi)$ is defined similarly but with `` $\inf$" replaced by `` $\sup$".   

 In particular, the  $L_{p}$ geominimal surface area of $f=e^{-\psi}$ can be defined as  $$G_p(f)=G_{p, e^{-t}, e^{-t}}(\psi).$$  
 \end{defn}    
 
 Results in subsections \ref{section 3.2}  and \ref{section 3.3}  can be modified accordingly to  $G_{p, F_1, F_2}(\psi)$ and $G_{p}(f).$ For instance, $G_{p, F_1, F_2}(\psi)$ is $SL_{\pm}(n)$-invariant. Moreover,  for $T\in GL(n)$,   $$G_{p, F_1, F_2}(\psi\circ T)=\big|\det (T)\big|^{\frac{p-n}{p+n}}\cdot  G_{p, F_1, F_2}(\psi).$$ It also has the homogeneous degree $\frac{n(p-n)}{p+n}$, i.e.,  $$G_{p, F_1, F_2}(\psi\circ \lambda)=|\lambda|^{\frac{n(p-n)}{p+n}} \cdot G_{p, F_1, F_2}(\psi),$$ where $(\psi\circ \lambda)(x)=\psi(\lambda x)$ for $\lambda\in \bbR$ and $x\in \bbR^n$. 
 
 Let ${F} \colon \R\to (0, \infty)$  satisfy that  $F\big(\frac{\|\cdot\| ^2}{2}\big)$  is log-concave and $0<I(F, 1) <\infty. $ Corollary \ref{equality of ball-1} implies that for $-n\neq p\in \bbR$ and $c>0$ a constant,   \begin{equation}\label{Ellipsoid-p--1}G_{p, F, F}\Big(\frac{c^2\|\cdot\|^2}{2} \Big) =c^ {\frac{ n(p-n)}{n+p}}\cdot    I\big(F, 1\big),\end{equation}  and in particular  $G_p\big(\gamma_n\circ c\big)=c^{\frac{n(p-n)}{p+n}} (\sqrt{2\pi})^{n}$. 
 
  A direct consequence of Proposition \ref{bounded by volume product} is the following result. Similar inequalities were obtained in  \cite{Caglar-6, CaglarWerner}.
\begin{prop}\label{bounded by volume product-1} Let $\psi\in \C$. If $F_2\circ \psi^*\in \M$, then for $p\in (0, \infty)$ and $f=e^{-\psi}$, \begin{eqnarray*}   G_{p, F_1, F_2}(\psi)    &\leq&   {\big[I(F_1\circ \psi, \psi)\big]^{\frac{n}{n+p}} \cdot  \big[I(F_2\circ \psi^*, \psi^*) \big]^{\frac{p}{n+p}}}, \\ G_p(f)   &\leq&   {\big[I(f)\big]^{\frac{n}{n+p}} \cdot  \big[I(f^\circ) \big]^{\frac{p}{n+p}}}.\end{eqnarray*}  
 Similar inequalities hold  for $p\in (-\infty, -n)\cup (-n, 0)$  with `` $\leq$" replaced by `` $\geq$".  \end{prop}

 Immediately from Theorem \ref{generalaffineineq},  one has the following functional $L_p$ affine isoperimetric inequalities.  
\begin{theo}\label{generalaffineineq--1}  Let $F_1, F_2 \colon \R \to (0, \infty)$ be measurable functions and  $\psi \in \C_0$ such that $F_2\circ \psi^*$ and $\breve{F}\big(\frac{\|\cdot\| ^2}{2}\big)$ are log-concave. Assume that $0<I(\breve{F}, 1)<\infty$. 
 
  \noindent  (i)  Let  $p>0$. 
Then,   \[ \frac{G_{p, F_1, F_2}(\psi) }{G_{p, \breve{F}, \breve{F}}\big(\frac{\|\cdot\|^2}{2} \big)} \leq \min\bigg\{\bigg(\frac{I(F_2\circ \psi^*,  \psi^*)}{I(\breve{F}, 1)}\bigg)^{\frac{p-n}{p+n}},\ \  \bigg(\frac{I(F_1\circ\psi,  \psi)}{I(\breve{F}, 1)}\bigg)^{\frac{n-p}{n+p}}\bigg\}.\]    
 (ii) Let  $p\in (-n, 0)$. 
Then,   \[  \frac{G_{p, F_1, F_2}(\psi) }{G_{p, \breve{F}, \breve{F}}\big(\frac{\|\cdot \|^2}{2} \big)}  \geq  \bigg(\frac{I(F_1\circ\psi,  \psi)}{I(\breve{F}, 1) }\bigg)^{\frac{n-p}{n+p}}.
\]    (iii) Let  $p<-n$. 
Then,  \[ \frac{G_{p, F_1, F_2}(\psi) }{G_{p, \breve{F}, \breve{F}}\big(\frac{\|\cdot\|^2}{2} \big)} \geq  \bigg(\frac{I(F_2\circ \psi^*,  \psi^*)}{I(\breve{F}, 1)}\bigg)^{\frac{p-n}{p+n}}.
\]  If $\breve{F}$ is strictly decreasing, equality holds in each case   if and only if  $\psi, \breve{F},  F_1$ and $F_2$ satisfy formula (\ref{equality:character}).
  \end{theo}  
  In particular, for the $L_p$ geominimal surface area of log-concave functions, one has the following functional $L_p$ affine isoperimetric inequality. Similar inequalities were obtained in  \cite{Caglar-6, CaglarWerner}.

  \begin{cor}\label{generalaffineineq--1-logconcave}  Let $\psi \in \C_0$ and $f=e^{-\psi}$.  
 
  \noindent  (i)  Let  $p>0$.  
Then,   \[ \frac{G_p(f) }{G_p(\gamma_n)} \leq \min\bigg\{\bigg(\frac{I(f^\circ)}{I(\gamma_n)}\bigg)^{\frac{p-n}{p+n}},\ \  \bigg(\frac{I(f)}{I(\gamma_n)}\bigg)^{\frac{n-p}{n+p}}\bigg\}.\]     (ii) Let  $p\in (-n, 0)$. 
Then,   \[   \frac{G_p(f) }{G_p(\gamma_n)} \geq   \bigg(\frac{I(f)}{I(\gamma_n)} \bigg)^{\frac{n-p}{n+p}}.\]   (iii) Let  $p<-n$. 
Then,  \[  \frac{G_p(f) }{G_p(\gamma_n)}   \geq  \bigg(\frac{I(f^\circ)}{I(\gamma_n)} \bigg)^{\frac{p-n}{p+n}}.
\]  Equality holds in each case if and only if  $\psi$ satisfies formula (\ref{equality:character-logconcave}).
  \end{cor}  
  
 In the following theorem, we provide a Santal\'{o} type inequality for the general  $L_p$ geominimal surface area of convex functions. It is a generalization of inequality (\ref{eq:BSfunctional}). 
\begin{theo}\label{generalaffineineq--2} Let $F_1, F_2 \colon \R \to (0, \infty)$ be measurable functions and  $\psi\in \C_0$ such that  $F_1\circ \psi$, $F_2\circ \psi ^*$ and $\breve{F}\big(\frac{\|\cdot\| ^2}{2}\big)$ are log-concave. Assume that $0<I(\breve{F}, 1)<\infty$.    Then,  for $p>0$, \[G_{p, F_1, F_2}(\psi) \cdot  G_{p, F_2, F_1}(\psi^*)  \leq \Big[G_{p, \breve{F}, \breve{F}}\Big(\frac{\|\cdot \|^2}{2} \Big)\Big]^2.
\]  If $\breve{F}$ is strictly decreasing, equality holds  if and only if  $\psi, \breve{F},  F_1$ and $F_2$ satisfy formula (\ref{equality:character}).\end{theo} 

\begin{proof}  For $p>0$, by Proposition \ref{bounded by volume product-1} and inequality  (\ref{eq:BSfunctional}), one has, \begin{eqnarray*}  G_{p, F_1, F_2}(\psi) \cdot  G_{p, F_2, F_1}(\psi^*)  & \leq&   I(F_1\circ \psi, \psi)  \cdot   I(F_2\circ  \psi^*, \psi^*)   \\ &\leq&  \big[I(\breve{F}, 1)\big] ^2= \Big[G_{p, \breve{F}, \breve{F}}\Big(\frac{\|\cdot \|^2}{2} \Big)\Big]^2, \end{eqnarray*}  where the last equality follows from formula (\ref{Ellipsoid-p--1}).  The characterization of equality follows along the same lines as in Theorem \ref{generalaffineineq}.  \end{proof}

More generally,   if  $h\in \Phi$ such that $h(t)h(s)\leq  [h(r)]^2 $ for all $r, s, t> 0$ satisfying $st\geq r^2$, then   \begin{eqnarray*}
 \OrliczA(\psi) \cdot as^{orlicz}_{h, F_2, F_1} (\psi^*)  \leq \Big[as^{orlicz}_{h, \breve{F}, \breve{F}}\Big(\frac{\|\cdot\|^2}{2}\Big) \Big]^2, \end{eqnarray*}  and if in addition $F_1\circ \psi$, $F_2\circ \psi ^*$ and $\breve{F}\big(\frac{\|\cdot\| ^2}{2}\big)$ are log-concave, 
  \begin{eqnarray*}
 \OrliczG(\psi) \cdot G^{orlicz}_{h, F_2, F_1} (\psi^*)  \leq \Big[G^{orlicz}_{h, \breve{F}, \breve{F}}\Big(\frac{\|\cdot \|^2}{2}\Big) \Big]^2. \end{eqnarray*}

 Moreover,  the following Santal\'{o} type  inequality for log-concave functions holds. These results extend the functional Blaschke-Santal\'{o} and inverse Santal\'{o} inequality \cite{ArtKlarMil, KBallthesis, FradeliziMeyer2007, FradeliziMeyer2008,    KlartagMilman, Lehec2009}. Similar inequalities were obtained in  \cite{Caglar-6, CaglarWerner}.

 \begin{cor} \label{santalo--=1} Let $\psi\in \C_0$ and $f=e^{-\psi}$. 
 
  \noindent  (i) For $p\in (0, \infty)$,  the following inequality holds, with equality if and only if  $\psi$ satisfies formula (\ref{equality:character-logconcave}), \[ G_p(f) \cdot G_p(f^\circ)  \leq \big[{G_p(\gamma_n)}\big]^2=(2\pi)^{n}.
\]   
 (ii) For $p\in (-\infty, -n)\cup (-n, 0)$, there  is a universal constant $C>0$, such that,   \[G_p(f) \cdot G_p(f^\circ)   \geq  C^n\cdot \big[{G_p(\gamma_n)}\big]^2.
\]    \end{cor} \begin{proof} The part (i) follows immediately from Theorem \ref{generalaffineineq--2}.  Now let $p\in (-\infty, -n)\cup(-n, 0)$. By Proposition \ref{bounded by volume product-1}, one has, $$G_p(f) \cdot G_p(f^\circ)    \geq   I(f)  \cdot   I(f^\circ)  \geq   C^n\cdot (2\pi)^{n}= C^n\cdot \big[{G_p(\gamma_n)}\big]^2,$$  where the second inequality follows from the functional inverse Santal\'{o} inequality  \cite{FradeliziMeyer2008, KlartagMilman}.     \end{proof}

 \noindent {\bf Remark.} Let $h\in \Phi$ be such that $h(t)h(s)\leq  [h(r)]^2 $ for all $r, s, t> 0$ satisfying $st\geq r^2$, then   \begin{eqnarray*}
G^{orlicz}_{h}(f)\cdot G^{orlicz}_{h}(f ^\circ)\leq \big[G^{orlicz}_{h}\big(\gamma_n\big)\big]^2;\end{eqnarray*} while, if $h\in \Psi$ satisfying $h(t)h(s)\geq A\cdot [h(r)]^2$ for some constant $A>0$ and for all $r, s, t> 0$ satisfying $st\geq r^2$, then, there is a universal constant $C>0$, such that  \begin{eqnarray*}
G^{orlicz}_{h}(f)\cdot G^{orlicz}_{h}(f ^\circ) \geq A C^n\cdot \big[G^{orlicz}_{h}\big(\gamma_n\big)\big]^2. \end{eqnarray*}

 
\section{Orlicz affine and geominimal surface areas for $s$-concave functions }\label{section 4} 
 
 Let $s \in (0, \infty)$.  A nonnegative function  $f$  is $s$-concave if $f^s$ is concave on its support \cite{Borell1975}, that is, for all $\lambda \in [0,1]$ and for all $x, y\in \bbR^n$ such that $f(x)>0$ and $f(y)>0$, one has,
\[
f( (1-\lam)x + \lam y) \ge \big( (1-\lam) f(x)^s + \lam f(y)^s \big)^{1/s}.
\]  The support set of $f$ is $S_f:=\{x: f(x) >0\}$. Note that $S_f$ is a convex set  in $\bbR^n$. Throughout this section, assume that $S_f$ is open and bounded with $0\in \bbR^n$ in the interior of $S_f$, and $\lim_{x\rightarrow \partial S_f} f(x)=0$ where $\partial S_f$ is the boundary of $S_f$.   Let $ \Cs$ be the collection of all   upper semi-continuous  $s$-concave  functions whose supports satisfy above assumptions.    Define the function $\psi$ on $S_f$  by   
\begin{equation}\label{def:psi}
\psi(x)= \frac{1-f^s(x)}{s} \ \ \Leftrightarrow \ \ f(x)=\big(1-s\psi(x)\big)^{\frac{1}{s}},   \  \ \  \forall x\in S_f.
 \end{equation}
Note that $\psi$ is well defined and is convex on $S_f$. Moreover,  $1 - s \psi = f^s > 0$ and hence $\psi(x)<\frac{1}{s}$ for all $x\in S_f$. In the later context, the pair of functions $(f, \psi)$  refers to $f\in \Cs$ and its associated convex function $\psi$  by formula (\ref{def:psi}).   The following dual  function  $\ps$ for convex function $\psi$  is crucial in this section: \begin{equation}
\label{def: psi-stern}
\ps (y) =  \sup_{x \in S_f} \frac{\langle x, y \rangle - \psi(x)}{1 - s \psi(x)}.   
\end{equation}  It is easily checked that $\ps$ is convex and  $(\ps)_{(s)}^{\star} = \psi$ for all $f\in \Cs$. With the help of $\ps$, one can define  the $(s)$-Legendre dual of $f\in \Cs$ by
\[
f_{(s)}^\circ(y) = \big(1- s \ps(y)\big)^{1/s}, \quad \ \ \forall y \in S_{f_{(s)}^\circ}=\big \{ y:  1 - s \ps(y) > 0\big\}.
\] Equivalently (which coincides with the definition introduced in \cite{ArtKlarMil, ArtMil}), by letting $a_+=\max\{a, 0\}$,   $$f_{(s)}^\circ(y)  = \inf_{x \in S_f}  \frac{\big[(1 - s \langle x, y \rangle)_+\big]^{1/s}}{f(x)}.$$ Note that $f_{(s)}^\circ$ is $s$-concave and upper semi-continuous.  Moreover, 
 $(f_{(s)}^\circ)_{(s)}^\circ = f$  and   $S_{f_{(s)}^\circ} = \frac{1}{s} S_f^\circ $ where $$S_f^\circ= \{ z: \langle x,z \rangle < 1\ \ \   \mathrm{for\ all}\ x \in S_f\}.$$  Throughout this section,   let $X_{\psi}\subset S_f$ be such that  $$X_\psi=\Big\{x\in S_f:  \nabla^2\psi, \ \mathrm{the \ Hessian\ matrix\ of} \ \psi\ \mathrm{in\  the\  sense\  of\  Alexandrov,\  exists\  and\  is\ invertible}\Big\}.$$ 
 
   For simplicity, let $\widetilde{\psi}(x)=1+s \langle x , \nabla \psi(x) \rangle -s\psi(x).$ 
 The supremum  in (\ref{def: psi-stern}) is attained if $x\in S_f$ and \begin{equation*}
\label{eq:attained}
y = \frac{1 - s \langle x, y \rangle}{1 - s \psi(x)} \, \nabla \psi(x) \ \ 
\hbox{or}\  \ \ y = (1 - s \ps (y)) \nabla \psi (x). 
\end{equation*}
This leads to $\langle x,y \rangle = \frac{1 - s \langle x, y \rangle}{1 - s \psi(x)} \, \langle x, \nabla \psi(x)\rangle$ and  \begin{equation}
\label{eq:gradient}
\frac{1}{1 - s \ps(y)} = \frac{1- s \psi(x)}{1 - s \langle x,y \rangle} = 1+s \langle x , \nabla \psi(x) \rangle -s\psi(x)=\widetilde{\psi}(x).
\end{equation} That is, the supremum  in (\ref{def: psi-stern}) is attained if $x\in S_f$ and  
$y=
\frac{\nabla \psi(x)}{\widetilde{\psi}(x) }  = T_\psi (x).$ Moreover,  \begin{eqnarray} \ps (y) = \frac{\langle x, y \rangle - \psi(x)}{1 - s \psi(x)}, \ \ \ \ \ \  \,dy=\frac{1 - s \psi(x)}{\big( \widetilde{\psi}(x) \big)^{n+1}} \ \det \Hess \psi(x) \ dx. \label{eq:Jacobian2} \end{eqnarray}  See  \cite{Caglar-6} for details.  Similar to the proof of Theorem 4 in \cite{Caglar-6}, for an integrable function $g$ defined on $X_{\ps}$,  one has, 
 \begin{equation}\label{definition for I_s(g)}
I_s(g, \psi_{(s)}^\star)=\int_{X_{\ps}}\!\! g (y)\, dy =  \int_{X_{\psi}}\! g \big(T_\psi (x) \big)  \cdot  \frac{(1-s\psi(x)) \cdot \det   \Hess  \psi (x) }{(\widetilde{\psi}(x))^{n+1}} \ dx.
\end{equation}  

\subsection{Definition and Properties}
Let $h: (0, \infty)\rightarrow (0, \infty)$ be a continuous function.  For simplicity, let $\Fs$ be the set of all positive integrable functions defined on $X_{\psi^{\star}_{(s)}}$, i.e., for all $g\in \Fs$, one has $g(y)>0$ for all $y\in X_{\psi^{\star}_{(s)}}$ and $0<I_s(g, \psi_{(s)}^\star)<\infty.$  Let $(f, \psi)$ be the pair given by  formula (\ref{def:psi}) and $f\in \Cs$.

 \begin{defn} \label{Orlicz-mixed-integral-s-concave-1}   The Orlicz  $L^{(s)}_{h}$-mixed integral of $\psi$ and  $g\in\Fs$ is defined by  
\begin{eqnarray*} 
V_{h}^{(s)}(\psi, g) =  \int_{X_{\psi}}  h \Big(g(T_{\psi}(x))  (\widetilde{\psi}(x))^{(\frac{1}{s} -1)} (1-s\psi(x))  \Big) \cdot \widetilde{\psi}(x) \cdot (1-s\psi(x))^{(\frac{1}{s}-1)} \,dx. \end{eqnarray*} 
 \end{defn}   It can be proved, similar to the proof of Lemma \ref{affine-invariance:mixed-1},  that for all $T\in SL_{\pm}(n)$, one has, 
$$V_{h}^{(s)} (\psi \circ T, g\circ T^{-t})=V_{h}^{(s)} (\psi, g).$$ 
 We write $V_{p}^{(s)}(\psi, g)$ for the case  $h(t)=t^{-p/n}$ with $-n\neq p\in \bbR$. 
 
 The following definition for the $L_p$ affine surface area of $s$-concave functions was given in \cite{Caglar-6}. 
\begin{defn}
\label{def:s}
Let $s > 0$ and the pair $(f, \psi)$ be given by formula (\ref{def:psi}). For any $-n\neq p \in \R$, the $L_p$ affine surface area of the $s$-concave function $f$ is defined by  
\[
as_{p}^{(s)}(\psi) = \frac{1}{1+ns} \  \int_{X_\psi} \frac{\left(1-s \psi(x) \right)^{\left(\frac{1}{s}-1\right)\cdot \frac{n}{n+p}}
 \left(\det   \Hess  \psi (x) \right)^{\frac{p}{n+p}}}
 {\big(\widetilde{\psi}(x)\big)^{{\frac{p}{n+p}}\big(n+\frac{1}{s}+1\big) - 1 }} \  dx. 
\]
\end{defn} 

 The following theorem provides a new formula for $as_{p}^{(s)}(\psi) $. 
\begin{theo}
  \label{equivalent:affine:surface:area2}  Let $s > 0$ and the pair $(f, \psi)$ be given by formula (\ref{def:psi}) with $f\in \Cs$.  	 Assume that $\psi$ is a $C^2$ strictly convex function. For $p\geq 0$,  one has \begin{equation*} as_{p}^{(s)}(\psi) =  \frac{1}{1+ns} \inf_{ g\in \Fs } \left\{ V_p^{(s)}(\psi, g)^{\frac{n}{n+p}}\
I_s(g, \psi_{(s)}^\star)^{\frac{p}{n+p}}\right\},
\end{equation*}  while for $-n\neq p<0$, the above formula holds with `` $\inf$" replaced by `` $\sup$". \end{theo}  
\begin{proof}  We only prove the case for $p\geq 0$ and the case for $-n\neq p<0$ follows similarly by the (reverse) H\"{o}lder's inequality. 
 It is clear that for $p=0$ and for all $g\in \Fs$  (see  \cite{Caglar-6} for details), 
\begin{eqnarray*}  V_{0}^{(s)} ( \psi, g)= \int _{X_{\psi}} \left(1-s \psi(x)  \right)^{(\frac{1}{s}-1)}\widetilde{\psi}(x)\,  dx=(1+ns)\cdot as_{0}^{(s)}(\psi).  \end{eqnarray*}   Let $p\in (0, \infty)$ and thus $\frac{p}{n+p}\in (0, 1)$. H\"older's inequality implies that for all function $g\in \Fs$,     
\begin{eqnarray*}
as_{p}^{(s)}(\psi)  
&=& \frac{1}{1+ns}    \int_{X_\psi}  \bigg[ \bigg(\frac{g^{-1} (T_{\psi}(x))}{  (\widetilde{\psi}(x))^{(\frac{1}{s} -1)} (1-s\psi(x)) }\bigg)^{\frac{p}{n}} \cdot \widetilde{\psi}(x) \cdot {(1-s\psi(x))^{(\frac{1}{s}-1)}} \bigg]^{\frac{n}{n+p}}   \nonumber \\
 && \ \ \times \left[  g ( T_{\psi}(x))  \cdot  \frac{(1-s\psi(x)) \cdot \det   \Hess  \psi (x) }{(\widetilde{\psi}(x))^{n+1}}  \right]^{\frac{p}{n+p}} \ dx  \nonumber  \\
&\leq&     \frac{1}{1+ns} V_p^{(s)}(\psi, g)^{\frac{n}{n+p}}\ I_s(g, \psi_{(s)}^\star)^{\frac{p}{n+p}}.   
\end{eqnarray*}
Taking the infimum over all  $g\in \Fs$, one gets, for all $p\in (0, \infty)$, 
\begin{equation*} as_{p}^{(s)}(\psi) \leq  \frac{1}{1+ns} \inf_{ g\in \Fs} \left\{ V_p^{(s)}(\psi, g)^{\frac{n}{n+p}}\
I_s(g, \psi_{(s)}^\star)^{\frac{p}{n+p}}\right\}.
\end{equation*}   Let $g_0$ be the function given by 
	\begin{equation*}
  g_0(T_{\psi}(x)) = \bigg(  \frac{(1-s\psi(x))^{(\frac{1}{s} -2 - \frac{p}{n}) } \cdot (\widetilde{\psi}(x))^{(n+2 -\frac{p}{ns} + \frac{p}{n})} }{ \det   \Hess  \psi (x)}  \bigg)^{\frac{n}{n+p}}.
	\end{equation*} 
	Then 
 $ I_s(g_0, \psi_{(s)}^\star) =(1+ns) as_{p}^{(s)}(\psi)=V_p^{(s)}(\psi, g_0)$  (see also Theorem 4 in \cite{Caglar-6}) and 
\begin{eqnarray*} as_{p}^{(s)}(\psi) &=&  \frac{1}{1+ns}  \left\{ V_p^{(s)}(\psi, g_0)^{\frac{n}{n+p}}\
I_s(g_0, \psi_{(s)}^\star)^{\frac{p}{n+p}}\right\} \\ &\geq&  \frac{1}{1+ns} \inf_{ g\in \Fs } \left\{ V_p^{(s)}(\psi, g)^{\frac{n}{n+p}}\
I_s(g, \psi_{(s)}^\star) ^{\frac{p}{n+p}}\right\}, \end{eqnarray*}
Thus the desired result follows. 

Now let us find an explicit expression for $g_0$.  Recall that $(\ps)_{(s)}^{\star} = \psi$, which implies  $T_{\psi}  \circ T_{\ps} =  {\mathrm{Id}}$ and $T_{\ps}  \circ T_{\psi} = {\mathrm{Id}}$. Hence, for $x\in X_{\psi}$ and  $y = T_\psi(x)$, one has 
\begin{equation}
\label{eq:JacobianDual}
\det \, \left(d_x T_{\psi}\right) \cdot \det \,  \big( d_{y} T_{\ps} \big)= 1.
\end{equation} Moreover,  for  $x\in X_{\psi}$ and $y=T_{\psi}(x)$,   equation (\ref{eq:gradient}) implies 
 \begin{equation}\label{duality-s-concave0001}
\frac{1}{1 - s \ps(y)} = \widetilde{\psi}(x) \ \ 
\hbox{and}\ \ 
\frac{1}{1 - s \psi(x)} =  1 + s ( \langle \nabla \ps(y), y \rangle - \ps(y))=\widetilde{\ps}(y).
\end{equation}  
Combining \eqref{eq:Jacobian2} with \eqref{eq:JacobianDual}, one gets  
\begin{equation*} 
\det \Hess \psi(x) \bigg( \frac{1 - s \ps(y)}{1 + s \langle \nabla \ps(y), y \rangle - s\ps(y)} \bigg)^{n+2} \det \Hess \ps(y) = 1.
\end{equation*} Thus, if $y=T_{\psi}(x)$, then
\begin{eqnarray*}
  g_0 ( y ) &=& \bigg(  \frac{(1 + s  \langle \nabla \ps(y), y \rangle - s\ps(y)) ^{(-\frac{1}{s} + \frac{p}{n} -n   ) } }{ ( 1 - s \ps(y)   )^{( \frac{p}{n}-\frac{p}{ns}  )}  }  \bigg)^{\frac{n}{n+p}} \  (\det \Hess \ps(y) )^{\frac{n}{n+p}}. \end{eqnarray*}  \end{proof} 

For $s>0$, let $k_s(x)= \big[\!\left(1- s \|x\|^2 \right)_{+}\!\big]^{\frac{1}{2s}}$. Note that $k_s (\cdot)$ is the special function which plays the role of the unit ``Euclidean ball" in $s$-concave functions, that is, $(k_s)^\circ_{(s)}=k_s$ (see e.g., \cite{ArtKlarSchuWer}). We also let  $$\int_{\{x\in \bbR^n: \|x\|<s^{-1/2}\}}k_s(x)\,dx= \left(\frac{\pi}{s}\right)^{\frac{n}{2}}  \frac{ \Gamma(1+\frac{1}{2s})}{ \Gamma(1+\frac{n}{2}+\frac{1}{2s})}=: \omega_{n,s}.$$ 

Motivated by Theorem \ref{equivalent:affine:surface:area2}, we now propose the following definition for the Orlicz affine and geominimal surface areas for $s$-concave functions.  Let $\Ms\subset \Fs$ be the  subset containing all log-concave functions. 

  \begin{defn} \label{Orlicz affine surface-s-concave}
		Let   $(f, \psi)$ be the pair given by formula  (\ref{def:psi}) with $f\in \Cs$.    For $h\in \Phi$,   the Orlicz $L^{(s)}_{h}$ affine surface area of  $\psi$  is defined by 
\begin{eqnarray*}
  \OrliczAS ( \psi) =\inf \left\{ V^{(s)}_{h} (\psi, g): g\in \Fs \ \mathrm{with} \ I_s(g, \psi_{(s)}^\star)=(1+ns)\cdot \omega_{n,s} \right\}, 
  \end{eqnarray*}  and the Orlicz $L^{(s)}_{h}$ geominimal surface area of  $\psi$  is defined by 
\begin{eqnarray*}
  \OrliczGS (\psi) =\inf \left\{ V^{(s)}_{h} (\psi, g): g\in \Ms \ \mathrm{with} \ I_s(g, \psi_{(s)}^\star)=(1+ns)\cdot \omega_{n,s} \right\}. 
  \end{eqnarray*}  When $h\in \Psi$,    the Orlicz $L^{(s)}_{h}$ affine and geominimal surface areas of  $\psi$ are defined as above with `` $\inf$" replaced by `` $\sup$". \end{defn} 
One can easily see that  both $\OrliczAS (\psi)$ and  $\OrliczGS (\psi)$ are $SL_{\pm}(n)$-invariant in the same fashion of Theorem \ref{Affine-invariance--1}.  It is clear that  $\OrliczAS(\psi) \leq \OrliczGS( \psi)$ for $h\in \Phi$ and $\OrliczAS(\psi)\geq \OrliczGS( \psi)$ for $h\in \Psi$.

Hereafter, for a constant $c>0$, let $\E^s_c(x)= \frac{1-[(1-sc^2\|x\|^2)_+]^{1/2}}{s}$. It can be checked that  $(\E^s_c)^{\star}=\E^s_{1/c}$.    The function $\E^s_c$ is associated to the $s$-concave function 
$k_s^c(x) = \big[(1- s c^2\|x\|^2)_+\big]^{\frac{1}{2s}}$ by identity (\ref{def:psi}). Note that $X_{\E^s_c}=\{x: \|x\|< c^{-1}s^{-1/2}\}$   and $X_{(\E^s_c)^{\star}}=\{y: \|y\|< cs^{-1/2}\}$.  
Moreover,   
 $$ I_s(k_s^c)=\int_{X_{\E^s_c}} k_s^c (x)\,dx =c^{-n} \cdot \omega_{n, s}.$$   Applying identity (\ref{identity001}) (which will be stated in the next subsection and was proved in \cite{Caglar-6}) to function $\E^s_c$, one has,  \begin{eqnarray}\label{identity001----1} (1+ns) \cdot I_s (k_s^c)= \int_{X_{\E^s_c}}   \widetilde{\E}^s_c(x)  \cdot \big(1-s\E^s_c(x)\big)^{(\frac{1}{s}-1)} \ dx= \int_{X_{\E^s_c}}   \left(1 - sc^2 \|x\|^2\right)_{+}^{(\frac{1}{2s} -1)} \, dx. \end{eqnarray}  The following corollary provides a precise value for the Orlicz affine and geominimal surface areas of $k_s^c$.  
  \begin{cor} \label{equality of ball-11}
Let $c>0$ be a constant. For all $h\in \Phi\cup\Psi$,
 $$ \OrliczAS\big( \E^s_c\big) = (1+ns) \cdot I_s\big(k_s^c \big) \cdot h (c^{-n})=(1+ns) \cdot c^{-n} \cdot \omega_{n, s} \cdot h (c^{-n}), $$ and if in addition  $\Big[\big(1 - \frac{s \|\cdot\|^2}{c^2} \big)_{+}\Big]^{(\frac{1}{2s} -1)}$ is a log-concave function (which holds if $s\leq 1/2$), $$ \OrliczGS\big( \E^s_c\big) = (1+ns) \cdot I_s\big(k_s^c \big) \cdot h (c^{-n})=(1+ns) \cdot c^{-n} \cdot \omega_{n, s} \cdot h (c^{-n}).$$
 \end{cor}

 \begin{proof}
 Note that  $\nabla \E^s_c (x)= \frac{c^2 x}{\big[(1-sc^2\|x\|^2)_+\big]^{1/2}}$ and $\big(1-s \E^s_c (x)\big) = \big[(1 - sc^2 \|x\|^2 )_+\big]^{1/2}$. Moreover, 
\begin{equation*}
 \widetilde{\E}^s_c(x) = 1+s \langle x , \nabla \E^s_c (x) \rangle - s \E^s_c (x) =  {\big[(1-sc^2\|x\|^2)_+\big]^{-1/2}}\ \  \ \ \ \mathrm{for} \  x\in X_{\E^s_c},
\end{equation*} which leads to $T_{\E^s_c}(x)=\frac{\nabla \E^s_c (x)}{ \widetilde{\E}^s_c(x) }=c^2x$.

 Applying Jensen's inequality to the convex function $h$ (as  $h\in \Phi$) and by formula (\ref{identity001----1}), one has,  
\begin{eqnarray*}
V_{h}^{(s)} (\E^s_c, g) \!\! &=&\!\!  \int_{X_{\E_c^s}}\!\! h \Big(g(T_{\E_c^s}(x))  (\widetilde{\E_c^s}(x))^{(\frac{1}{s} -1)} \big(1-s\E_c^s(x)\big)\Big) \cdot \widetilde{\E_c^s}(x) \cdot (1-s\E_c^s(x))^{(\frac{1}{s}-1)} \,dx\\  \!\!&\geq &\!\! (1+ns) \cdot  I_s(k_s^c)     \cdot  h\bigg(\int_{X_{\E^s_c}}  \frac{g(c^2x) \cdot  (\widetilde{\E}^s_c(x))^{\frac{1}{s}  } \cdot (1-s\E^s_c(x))^{\frac{1}{s}} }{(1+ns) \cdot I_s(k_s^c) }\, dx  \bigg) \\ & =& (1+ns) \cdot  I_s(k_s^c) \cdot  h \bigg(  \int_{X_{(\E^s_c)^{\star}}}  \frac{g(y) }{c^{2n}\cdot  (1+ns) \cdot I_s(k_s^c) } \,  dy \bigg)   \\
&=& (1+ns) \cdot  I_s(k_s^c) \cdot  h \bigg(   \frac{ I_s(g, \psi_{(s)}^\star)}{c^{2n}(1+ns) \cdot I_s(k_s^c) }  \bigg).
\end{eqnarray*}  
  This leads to, for all $h\in \Phi$, 
\begin{eqnarray*}
  \OrliczAS\big( \E^s_c\big) &=&\inf \left\{ V^{(s)}_{h} (\E_c^s, g): g\in \Fs \ \mathrm{with} \ I_s(g, \psi_{(s)}^\star)=(1+ns)\cdot \omega_{n,s} \right\}\\  &\geq&  (1+ns) \cdot I_s \big(k_s^c \big) \cdot h(c^{-n}). 
	\end{eqnarray*}   On the other hand,   \begin{eqnarray*}
 \OrliczAS\big( \E^s_c \big)   \leq     V_{h}^{(s)}\Big(\E^s_c,  c^{-n} \cdot \Big[\Big(1 - \frac{s \|\cdot \|^2}{c^2} \Big)_{+}\Big]^{(\frac{1}{2s} -1)}\Big) = (1+ns) \cdot I_s\big(k_s^c \big) \cdot h (c^{-n}), 
\end{eqnarray*} where we have used identity (\ref{identity001----1}) $$
 \int_{\{x\in \bbR^n: \ \|x\|\leq cs^{-1/2}\} } \Big[\Big(1 - \frac{s \|x\|^2}{c^2} \Big)_{+}\Big]^{(\frac{1}{2s} -1)}\,dx= c^{n}\cdot \omega_{n, s}\cdot (1+ns).
$$ Hence,  the desired formula follows.  Along the same lines, one gets the desired formula for $h\in\Psi$.

The proof of the geominimal case follows along the same lines if  $\big[\big(1 - \frac{s \|\cdot\|^2}{c^2} \big)_{+}\big]^{(\frac{1}{2s} -1)}$ is a log-concave function (holds if $s\leq 1/2$). This additional condition implies   \begin{eqnarray*}
 \OrliczGS\big( \E^s_c \big)   \leq     V_{h}^{(s)}\Big(\E^s_c,  c^{-n} \cdot \Big[\Big(1 - \frac{s \|\cdot \|^2}{c^2} \Big)_{+}\Big]^{(\frac{1}{2s} -1)}\Big) = (1+ns) \cdot I_s\big(k_s^c \big) \cdot h (c^{-n}),
\end{eqnarray*}  which provides the necessary upper bound for $\OrliczGS\big( \E^s_c \big)$. 
	\end{proof}  
	  \noindent  {\bf Remark.}  As one would expect, Corollary \ref{equality of ball-11} becomes   Corollary  \ref{equality of ball-1}
if  $s$ goes to $0$. Note that, when $s>1/2$, the function $\big[\big(1 - \frac{s \|\cdot\|^2}{c^2} \big)_{+}\big]^{(\frac{1}{2s} -1)}$ is not  log-concave (in fact log-convex). Hence, for $h\in \Phi$ and $s>1/2$, one only has \begin{eqnarray*}
  \OrliczGS\big( \E^s_c\big) \geq  (1+ns) \cdot I_s \big(k_s^c \big) \cdot h(c^{-n}).
	\end{eqnarray*}  This inequality holds for $h\in \Psi$ and $s>1/2$ with ``$\geq$" replaced by ``$\leq$".

 \subsection{Inequalities} In this subsection, we have additional assumptions for the $s$-concave function $f$, that is, $f$ is twice continuous differentiable on $S_f$, $\det \nabla^2 f\neq 0$ on $S_f$, $\lim_{x\rightarrow \partial S_f} f^s(x)=0$ and $0\in S_f$. The collection of all $s$-concave functions in $\Cs$ with the above addition conditions will be denoted by $\Cs^2$.  These assumptions imply that $X_{\psi}=S_f$ and $X_{\ps}=S_{f_{(s)}^\circ}$. Moreover, as showed in  \cite{Caglar-6},  for $f\in \Cs^2$,   \begin{eqnarray}\label{identity001}
(1+ ns) \cdot  I(f) = \int _{X_{\psi}} \left(1-s \psi(x)  \right)^{(\frac{1}{s}-1)}\widetilde{\psi}(x)\,  dx.  \end{eqnarray}   Consider the function $g_1$ as follows: \begin{eqnarray}
g_1(y) =  \big(1 - s \ps(y)\big)^{(\frac{1}{s} -1)} \cdot \big(1 + s \langle \nabla \ps(y), y \rangle -s \ps(y) \big). \label{equation:g1} 
\end{eqnarray} Let  $y = T_{\psi}(x)$. By formula (\ref{duality-s-concave0001}), one has,   
$$g_1 \big(T_{\psi}(x)\big)= (\widetilde{\psi}(x))^{(1-\frac{1}{s})} (1-s\psi(x)) ^{-1}.$$ By formulas (\ref{definition for I_s(g)}) and (\ref{identity001}), one has, (see also Theorem 4 in \cite{Caglar-6}) \begin{eqnarray*} 
I_s(g_1, \psi_{(s)}^\star) = \int_{X_{\ps}}  \!\!\! (1 - s \ps(y))^{(\frac{1}{s} -1)} \cdot (1 + s  \langle \nabla \ps(y), y \rangle - s\ps(y))\,dy =  (1+ns) \cdot I(f_{(s)}^\circ). 
\end{eqnarray*} 

The following result will be crucial in this subsection. Let $g_1$ be as in formula (\ref{equation:g1}). 
\begin{prop}
\label{bounded by volume product-s} 
Let  $(f,\psi)$ with  $f\in \Cs^2$ be  the pair given by formula (\ref{def:psi}).  Then, for all $h\in \Phi$,  
\begin{eqnarray*} 
	 \OrliczAS ( \psi) \leq   (1 + ns) \cdot I(f) \cdot h\bigg( \frac{\omega_{n, s} }{I (f_{(s)}^\circ)} \bigg), 
\end{eqnarray*}  and if in addition $g_1$ is log-concave, then 
\begin{eqnarray*} 
	 \OrliczGS ( \psi) \leq   (1 + ns) \cdot I(f) \cdot h\bigg( \frac{\omega_{n, s} }{I (f_{(s)}^\circ)} \bigg), 
\end{eqnarray*} 
Similar inequalities hold  for $h \in \Psi$  with `` $\leq$ " replaced by `` $\geq$ ".   \end{prop} 

\begin{proof}  We only prove the case for  $h\in \Phi$ and the proof for $h\in \Psi$ follows along the same line. Let $g_1$ be as in formula (\ref{equation:g1}).  In fact, for  $h\in \Phi$, \begin{eqnarray*} \OrliczAS (\psi) &=& \inf \left\{ V^{(s)}_{h} (\psi, g): g\in \Fs \ \mathrm{with} \ I_s(g, \psi_{(s)}^\star)=(1+ns)\cdot \omega_{n,s} \right\}   \\  &\leq&  \bigg\{ V^{(s)}_{h} \bigg(\psi, \frac{g_1 \cdot \omega_{n,s} }{I (f_{(s)}^\circ)}\bigg)\bigg\} =  (1 + ns) \cdot I(f) \cdot h\bigg( \frac{\omega_{n, s} }{I (f_{(s)}^\circ)} \bigg). \end{eqnarray*}    The results for $\OrliczGS (\psi)$ follows along the same lines if the additional assumption on $g_1$ is satisfied.  \end{proof}

 Proposition \ref{bounded by volume product-s}  becomes Proposition \ref{bounded by volume product} if  $s\rightarrow 0$.  Moreover,  the following cyclic inequalities for $\OrliczAS( \psi)$ and $\OrliczGS( \psi)$ hold whose proofs are similar to that for Theorem \ref{cyclic}. In fact, Theorem \ref{cyclic00} leads to Theorem \ref{cyclic} if $s\rightarrow 0$.
    
\begin{theo}\label{cyclic00} Let  $(f, \psi)$ be the pair given by (\ref{def:psi}) such that $f\in \Cs^2$.    Let $g_1$ be as in formula (\ref{equation:g1}). 

 \noindent
(i) Assume one of the following conditions: (a)   $h\in \Phi$ and $h_1\in \Psi$ with $H$ increasing; (b)   $h, h_1\in \Phi$ with $H$ decreasing; (c) $H$ concave increasing with either $h, h_1\in \Phi$ or $h, h_1\in \Psi$. Then
 \begin{equation*}
\frac{\OrliczAS(\psi)}{(1+ns)\cdot I(f)} \ \leq \ H\bigg(\frac{as_{h_1, s}^{orlicz}(\psi)}{(1+ns)\cdot I(f)} \bigg). 
\end{equation*}
(ii) Assume one of the following conditions: (d)  $h\in \Psi$ and $h_1\in \Phi$ with $H$ increasing;   (e) $H$ convex decreasing with one in $\Phi$ and the other one in $\Psi$; (f) $H$  convex increasing with either $h, h_1 \in \Phi$ or $h, h_1 \in \Psi$. Then 
\begin{equation*}
\frac{\OrliczAS(\psi)}{(1+ns)\cdot I(f)} \ \geq \ H\bigg(\frac{as_{h_1, s}^{orlicz}(\psi)}{(1+ns)\cdot I(f)} \bigg). 
\end{equation*} 

The same inequalities also hold for the Orlicz geominimal surface area, if in addition $g_1\in \Ms$ in conditions (a), (b) and (d). \end{theo}

Let $f\in \Cs$ and  $f_{z}(x)=\big(1-s\psi (x+z)\big)^{\frac{1}{s}}$ for $z\in \bbR^n$. Let $(f_{z})^{\circ}_{(s)}$ denote the  $(s)$-Legendre dual of $f_z$.   As proved in   \cite{Caglar-6, FradeliziMeyer2007}, there exists $z_0\in\R^n$ such that 
 \begin{eqnarray}
I(f_{z_0}) \cdot I\big((f_{z_0})_{(s)}^{\circ}\big) \leq  \left( \int_{\bbR^n} \big[\big(1-s \|x\|^2\big)_+\big]^\frac{1}{2s}dx \right) ^2=(\omega_{n, s})^2. \label{eq:Santalo2} \end{eqnarray}  Equality holds   in  (\ref{eq:Santalo2}) if and only if  there is  a positive definite matrix $A$ and a positive constant $ \widetilde{c}$, such that  $f_{z_0}(x) = \widetilde{c}\cdot  \big[\big(1-  s \|Ax\|^2\big)_+\big]^\frac{1}{2s}$. For simplicity, let $\Cs^{2, 0}$ be the collection of all $s$-concave functions in $\Cs^2$ such that $z_0=0$.

Let $c_s$ and $\bar{c}_s$ be constants defined by 
$$ c_s = \bigg(\frac{ I\big(f_{(s)}^\circ \big) }{ \omega_{n,s} } \bigg)^{\frac{1}{n}} \ \ \mathrm{and}  \ \ \bar{c}_s=\bigg(\frac{ \omega_{n,s}}{I(f)}\bigg)^{\frac{1}{n}}.$$

\begin{theo}\label{generalaffineineq1}
Let  $(f, \psi)$ be the pair given by formula (\ref{def:psi}) with $f\in \Cs^{2, 0}$.    

 \noindent  (i)   Assume that $0<c_s<\infty$. Then,  for   $h\in \Phi$, 
\begin{eqnarray*} 
	 \OrliczAS (\psi) \  \leq\  \OrliczAS\big(\E^s_{c_s}\big).	 \end{eqnarray*}  
 (ii) Assume that $0<\bar{c}_s<\infty$. Then, for $h\in \Phi$ be decreasing,   $$ \OrliczAS (\psi) \ \leq \    \OrliczAS(\E^s_{\bar{c}_s}).$$  The above inequality holds for $h\in \Psi$ with `` $\leq$" replaced by `` $\geq$".

 There is an equality in (i) and in (ii) if $h\in \Phi$ is strictly decreasing  (or  $h\in \Psi$ is strictly increasing) if and only if  $f(x) =  \widetilde{c} \big[\!\left(1-  s \|Ax\|^2\right)_+\big]^\frac{1}{2s}$ for some $\widetilde{c}>0$ and some  positive definite matrix $A$.
   \end{theo}

\begin{proof} 
(i). By inequality (\ref{eq:Santalo2}), one can check that $I(f)\leq \omega_{n, s} \cdot c_s^{-n}$.  Together with Proposition \ref{bounded by volume product-s}, one has, for all $h\in \Phi$,
 \begin{eqnarray*} 
	 \OrliczAS ( \psi) &\leq&  (1 + ns) \cdot I(f) \cdot  h\bigg(\frac{\omega_{n, s}}{I \big(f_{(s)}^\circ\big)}\bigg) \\
	 &\leq&   (1 + ns ) \cdot \omega_{n, s} \cdot c_s^{-n}   \cdot h \big(c_s^{-n}\big) \\ 
	&=& \OrliczAS\big(\E^s_{{c_s}}\big),\end{eqnarray*}  
where  the last equality is due to Corollary \ref{equality of ball-11}.   Equality holds in the above inequalities only if equality holds in inequality (\ref{eq:Santalo2}). That is, $f(x) = \widetilde{c} \big[\!\left(1-  s \|Ax\|^2\right)_+\big]^\frac{1}{2s}$.

On the other hand, assume that $f(x) = \widetilde{c} \big[\!\left(1-  s \|Ax\|^2\right)_+\big]^\frac{1}{2s}$. Then, equality holds in (\ref{eq:Santalo2}). Identity (\ref{identity001}) implies that $$\frac{\widetilde{c}}{\det A}=\frac{I(f)}{\omega_{n, s}}=\frac{I(f) \cdot I\big(f_{(s)}^\circ \big) }{\omega_{n, s}\cdot  I\big(f_{(s)}^\circ \big)}=\frac{\omega_{n, s} }{I\big(f_{(s)}^\circ \big)}=c_s^{-n}.$$ Let $\psi_0=\frac{1-f^s}{s}$. Then \begin{eqnarray*} \nabla \psi_0(x) = \frac{\widetilde{c}^s \cdot A^2x}{\big[\big(1-s\|Ax\|^2\big)_+\big]^{1/2}} \ \ \ \mathrm{and} \ \ \  \widetilde{\psi}_0(x)= \frac{\widetilde{c}^s}{\big[\big(1-s\|Ax\|^2\big)_+\big]^{1/2}}. \end{eqnarray*} Hence, $T_{\psi_0}(x)=A^2x$, and  \begin{eqnarray*}  V_{h}^{(s)} (\psi_0, g)=\int_{\{x: \|Ax\|<s^{-1/2}\}} h\Big(g(A^2x)\cdot \widetilde{c}\cdot \big[\big(1-s\|Ax\|^2\big)_+\big]^{1-\frac{1}{2s}}\Big)\cdot \widetilde{c}\cdot \big[\big(1-s\|Ax\|^2\big)_+\big]^{\frac{1}{2s}-1}\,dx.\end{eqnarray*} Similar to the proof of Corollary \ref{equality of ball-11}, let $g(A^2x)=(\det A)^{-1}{\big[\big(1-s\|Ax\|^2\big)_+\big]^{\frac{1}{2s}-1}}$, and then  \begin{eqnarray*} \OrliczAS ( \psi_0)&=&V_{h}^{(s)} (\psi_0, g)=\frac{\widetilde{c}}{\det A}\cdot h\Big(\frac{\widetilde{c}}{\det A}\Big)\cdot (1+ns)\cdot \omega_{n, s}\\&=&(1 + ns ) \cdot \omega_{n, s} \cdot c_s^{-n}   \cdot h \big(c_s^{-n}\big) \\ 
	&=& \OrliczAS\big(\E^s_{{c_s}}\big). \end{eqnarray*} In conclusion, equality holds if and only if $f(x) = \widetilde{c} \big[\!\left(1-  s \|Ax\|^2\right)_+\big]^\frac{1}{2s}$ for some constant $c>0$ and for some positive definite matrix $A$.

 \noindent 
(ii). By inequality (\ref{eq:Santalo2}), one can check that $I\big(f_{(s)}^\circ\big) \leq \omega_{n, s}  \cdot (\bar{c})^n .$ Together with Proposition \ref{bounded by volume product-s}, one has, for all decreasing $h\in \Phi$
 \begin{eqnarray*} 
	 \OrliczAS ( \psi) &\leq&  (1 + ns) \cdot I(f) \cdot  h\bigg(\frac{\omega_{n, s}}{I \big(f_{(s)}^\circ\big)}\bigg) \\
	 &\leq&   (1 + ns ) \cdot \omega_{n, s} \cdot (\bar{c}_s)^{-n} \cdot  h\big( (\bar{c}_s)^{-n} \big) \\
	&=& \OrliczAS( \E^s_{\bar{c}_s}),\end{eqnarray*} 
 \noindent
where the last equality is due to Corollary \ref{equality of ball-11}.   

Similar to the characterization of equality in (i), one can prove that if $h$ is strictly decreasing, equality holds in the above inequality if and only if equality holds in inequality (\ref{eq:Santalo2}), that is, $f(x) = \widetilde{c}\big[\! \left(1-  s \|Ax\|^2\right)_+\big]^\frac{1}{2s}$. 

Similarly, the desired result for $h\in \Psi$ holds if   `` $\leq$ " is replaced by `` $\geq$ ".  
\end{proof}  

Let $g_1$ be as in formula (\ref{equation:g1}).  If $g_1$ is a log-concave function and $0<c_s<\infty$, then for  $h\in \Phi$,   $$
	 \OrliczGS (\psi) \leq (1 + ns ) \cdot \omega_{n, s} \cdot c_s^{-n}   \cdot h \big(c_s^{-n}\big).$$ Moreover, If $g_1$ is a log-concave function and $0<\bar{c}_s<\infty$, for $h\in \Phi$ being decreasing,   $$ \OrliczGS (\psi) \leq (1 + ns ) \cdot \omega_{n, s} \cdot (\bar{c}_s)^{-n} \cdot  h\big( (\bar{c}_s)^{-n} \big);$$  while for $h\in \Psi$, the above inequality holds with `` $\leq$" replaced by `` $\geq$".  These inequalities together with Corollary \ref{equality of ball-11} and its remark imply the following result.  
\begin{cor}\label{generalaffineineq1==11=1}
Let  $(f, \psi)$ be the pair given by formula (\ref{def:psi}) with $f\in \Cs^{2, 0}$ and let $g_1$ be log-concave.  

 \noindent  (i)  Assume that $0<c_s<\infty$. Then,  for $h\in \Phi$,   
\begin{eqnarray*} 
	 \OrliczGS (\psi) \  \leq\  \OrliczGS\big(\E^s_{c_s}\big).	 \end{eqnarray*}  
 (ii)  Assume that $0<\bar{c}_s<\infty$. Then, for $h\in \Phi$ being decreasing,  $$ \OrliczGS (\psi) \ \leq \    \OrliczGS(\E^s_{\bar{c}_s}).$$  The above inequality holds for $h\in \Psi$ with `` $\leq$" replaced by `` $\geq$". 
 
 If $s\leq 1/2$, there is an equality in (i) and in (ii) if $h\in \Phi$ is strictly decreasing  (or  $h\in \Psi$ is strictly increasing) if and only if  $f(x) =  \widetilde{c} \big[\!\left(1-  s \|Ax\|^2\right)_+\big]^\frac{1}{2s}$ for some constant $\widetilde{c}>0$ and some  positive definite matrix $A$.
   \end{cor} 

 Note that Theorem \ref{generalaffineineq1} and Corollary  \ref{generalaffineineq1==11=1} would become Corollary \ref{generalaffineineq-2} if  $s$ goes to zero.

 \subsection{The $L_p$ geominimal surface area of $s$-concave functions and a Santal\'{o} type inequality}

 The $L_p$ affine surface area of $s$-concave functions was investigated in \cite{Caglar-6}. In this subsection, we will briefly discuss the properties for the $L_p$ geominimal surface area of $s$-concave functions. Taking Theorem \ref{equivalent:affine:surface:area2} into account, it is more natural to define $G^{(s)}_p(\psi)$, the $L_p$ geominimal surface area of the $s$-concave function $f$, for  $-n\neq p\in \bbR$,  as follows.  
  \begin{defn}
  \label{equivalent:affine:surface:area-homogeneous p case-s} Let   $(f, \psi)$ be the pair given by formula  (\ref{def:psi}) with $f\in \Cs$.   For $p\geq 0$,  define   
	\begin{eqnarray*} G_{p}^{(s)} (\psi)&=& \Big(\frac{1}{1+ns} \Big)  \cdot \inf_{ g\in \Ms } \left\{ V_p^{(s)}(\psi, g)^{\frac{n}{n+p}}\ 
I_s(g, \psi_{(s)}^\star)^{\frac{p}{n+p}}\right\}  \\
&=& (\omega_{n,s})^{\frac{p}{n+p}}   \cdot   \bigg(\frac{\OrliczGS (\psi)}{1+ns} \bigg)^{\frac{n}{n+p}}, 
\end{eqnarray*}   with $h(t)=t^{-p/n}$. For $-n\neq p<0$, $G_{p}^{(s)} (\psi)$ is defined similarly but with `` $\inf$" replaced by `` $\sup$".   \end{defn}  
  The results in previous subsections can be modified accordingly to the $L_p$ geominimal surface area. In particular, it is $SL_{\pm}(n)$-invariant with homogeneous degree $\frac{n(p-n)}{p+n}$. If $c>0$ is a constant and $s\leq 1/2$,  Corollary \ref{equality of ball-11} implies that for all $-n\neq p\in \bbR$, 
 \begin{equation*}
 G_{p}^{(s)}\big(\E^s_c\big) =\ c^{\frac{ n(p-n)}{n+p}} \cdot \omega_{n,s}.
\end{equation*}  Moreover, the remark of Corollary \ref{equality of ball-11} implies that if $s>1/2$, then for $p>0$,  \begin{equation*}
 G_{p}^{(s)}\big(\E^s_c\big) \geq \ c^{\frac{ n(p-n)}{n+p}} \cdot \omega_{n,s},
\end{equation*} and for $-n\neq p<0$,  \begin{equation*}
 G_{p}^{(s)}\big(\E^s_c\big) \leq \ c^{\frac{ n(p-n)}{n+p}} \cdot \omega_{n,s}.
\end{equation*}

 A direct consequence of Proposition \ref{bounded by volume product-s}  is the following result, which leads to Proposition \ref{bounded by volume product-1} if  $s\rightarrow 0$. Similar inequalities were obtained in  \cite{Caglar-6, CaglarWerner}. Let $g_1$ be as in formula (\ref{equation:g1}).

\begin{prop}\label{bounded by volume product-11} Let  $(f, \psi)$ be the pair given by formula (\ref{def:psi}) with $f\in \Cs^2$ and  $g_1\in \Ms$. Then,  
\begin{eqnarray*}   G_{p}^{(s)} (\psi) \ \leq  \  \big[I(f)\big]^{\frac{n}{n+p}} \cdot  \big[I( f_{(s) }^\circ) \big]^{\frac{p}{n+p}},\end{eqnarray*}  for all $p\geq 0$. 
 Similar inequalities hold  for $p\in (-\infty, -n)\cup (-n, 0)$  with `` $\leq$" replaced by `` $\geq$".  \end{prop} 
 
 Suppose that $g_1$ and  $g_2$ are log-concave with $g_1$ as in formula (\ref{equation:g1}) and   \begin{eqnarray*}
g_2(y) =  \big(1 - s \psi(y)\big)^{(\frac{1}{s} -1)} \cdot \big(1 + s \langle \nabla \psi(y), y \rangle -s \psi(y) \big).  
\end{eqnarray*}    Proposition \ref{bounded by volume product-11} implies that   for $p\geq 0$,  $$ G_{p}^{(s)} (\psi) \cdot G_{p}^{(s)}(\ps)  \leq I(f)  \cdot   I(f_{(s)}^\circ),$$   while  for $-n\neq p<0$  the above inequality holds with `` $\leq$" replaced by `` $\geq$". If in addition $f\in \Cs^{2, 0}$,  the following Santal\'{o} type inequality for s-concave functions holds:   for $p>0$,   \[ G_{p}^{(s)} (\psi) \cdot G_{p}^{(s)}(\ps)  \leq  \omega_{n, s}^2 \leq  \big[ G_{p}^{(s)} (\E^s_1)\big]^2. \]  Moreover, if $s\leq 1/2$, there is an equality  if and only if  $f(x) = \widetilde{c} \big[\!\left(1-  s \|Ax\|^2\right)_+\big]^\frac{1}{2s}$ for some $\widetilde{c}>0$ and positive definite matrix $A$. 
   
 Immediately from Proposition \ref{bounded by volume product-11} and inequality (\ref{eq:Santalo2}),  one has the following functional $L_p$ affine isoperimetric inequalities for $s$-concave functions, which  becomes Corollary \ref{generalaffineineq--1-logconcave} if $s\rightarrow 0$. Similar inequalities were obtained in  \cite{Caglar-6, CaglarWerner}. Let $g_1$ be as in formula (\ref{equation:g1}). 
 
\begin{cor}\label{generalaffineineq--11} Let  $(f, \psi)$ be the pair given by formula (\ref{def:psi}) with $f\in \Cs^{2, 0} $ and  $g_1\in \Ms$. 
Assume that $0<I(f)<\infty$ and $0<I(f_{(s) }^\circ)<\infty$. 

 \noindent  
(i)  Let  $p>0$. 
Then,   \[\frac{G_{p}^{(s)} (\psi)}{G_{p}^{(s)} (\E^s_1)}  \leq \min\bigg\{\bigg(\frac{I( f_{(s)}^\circ )}{\omega_{n, s}}\bigg)^{\frac{p-n}{p+n}},\ \   \bigg(\frac{I(f)}{\omega_{n, s}}\bigg)^{\frac{n-p}{n+p}}\bigg\}.\]  
 (ii) Let  $p\in (-n, 0)$. 
Then,   \[\frac{G_{p}^{(s)} (\psi)}{G_{p}^{(s)} (\E^s_1)}  \geq  \bigg(\frac{I(f)}{\omega_{n, s}}\bigg)^{\frac{n-p}{n+p}}.\]  
(iii) Let  $p<-n$. 
Then,   \[\frac{G_{p}^{(s)} (\psi)}{G_{p}^{(s)} (\E^s_1)}  \geq  \bigg(\frac{I( f_{(s)}^\circ )}{\omega_{n, s}}\bigg)^{\frac{p-n}{p+n}}.\]   

 Moreover, if $s\leq 1/2$, there is an equality  if and only if  $f(x) = \widetilde{c} \big[\!\left(1-  s \|Ax\|^2\right)_+\big]^\frac{1}{2s}$ for some $\widetilde{c}>0$ and positive definite matrix $A$. \end{cor}

   \section{The general mixed Orlicz affine and geominimal surface areas for multiple convex functions}\label{section:mixed}
   
   In this section, we introduce the general  mixed Orlicz affine and geominimal surface areas for multiple convex functions. We have this notion only for convex functions in this section, but one can introduce it  for s-concave functions as well along the same lines. 
   
   Let  $\vec h = (h_1, \cdots, h_m)$, $\vec g = (g_1, \cdots, g_m)$, $\vec F^1 = (F_1^1, F_2^1, \cdots, F_m^1)$, $ \vec F^2 = (F_1^2, F_2^2, \cdots, F_m^2)$ etc.   We say $\vec h \in \Phi^m$ (or $\vec h \in \Psi^m$) if each $h_i \in \Phi$ (or $h_i \in \Psi$).  Assume that $X_{\vec{\psi}}=\cap_{j=1}^{m} X_{ \psi_j}$  is a nonempty set. Define
   \begin{eqnarray*}
   V_{\vec h, \vec F^1, \vec F^2} (\vec \psi, \vec g)= \int _{X_{\vec{\psi}}}  \prod_{i=1}^{m} \bigg[ h_i \bigg(\frac{g_i (\nabla \psi_i (x))}{F_i^2 ( \langle x, \nabla \psi_i(x) \rangle - \psi_i(x))}\bigg) F_i^1 (\psi_i(x)) \bigg]^{\frac{1}{m}}\,dx. 
   \end{eqnarray*}   If $\psi_i = \psi$, $g_i = g$,  $h_i = h$, $F_i^1 = F_1$, and $F_i^2 = F_2$ for all $1 \leq i \leq m$, then $V_{\vec h, \vec F^1, \vec F^2} (\vec \psi, \vec g)$ becomes $V_{h, F_1, F_2}(\psi,  g)$  in Definition \ref{Orlicz-mixed-integral-log-concave-1}.
   
  The general  mixed Orlicz  affine and geominimal surface areas for multiple convex functions are defined as follows. Note that there are many different ways to  define mixed Orlicz affine  and geominimal surface areas,  but we only focus on the one introduced below  due to high similarity of their properties. 
   
   \begin{defn} \label{Mixed Orlicz affine surface}  Let  $F_i^1, F_i^2 \colon \R\to (0, \infty)$ be measurable functions and $\psi_i \in \C$ for $1 \leq i \leq m$. For $\vec h \in \Phi^m$,   the general  mixed Orlicz  affine surface area of  $\vec \psi$  is defined by 
   \begin{eqnarray*}
     \OrliczAmix(\vec\psi) = \inf \left\{ V_{\vec h,  \vec F^1, \vec F^2} (\vec \psi, \vec g):   g_i \in \mathcal{F}^+_{\psi_i^*} \ \mathrm{with}\ I(g_i, \psi_i^*)=(\sqrt{2\pi})^n,\ \  1\leq i\leq m \right\}, 
     \end{eqnarray*} and  the general  mixed Orlicz   geominimal surface area of  $\vec \psi$  is defined by 
   \begin{eqnarray*}
     \OrliczGmix(\vec\psi) = \inf \left\{ V_{\vec h,  \vec F^1, \vec F^2} (\vec \psi, \vec g):   g_i \in \mathcal{L}_{\psi_i^*} \ \mathrm{with}\ I(g_i, \psi_i^*)=(\sqrt{2\pi})^n, \ \ 1\leq i\leq m \right\}. 
     \end{eqnarray*}
 The general  mixed Orlicz  affine and geominimal surface areas of  $\vec \psi$ for $\vec h\in \Psi^m$  are defined similarly with `` $\inf$" replaced by `` $\sup$". 
   \end{defn}

 As before,  one can check that $\OrliczAmix(\vec\psi) \leq \OrliczGmix(\vec\psi)$  for $\vec h \in \Phi^m$ and $\OrliczAmix(\vec \psi) \geq \OrliczGmix(\vec \psi)$  for $\vec h \in \Psi^m$.  If $F_i^1=F_i^2=e^{-t}$ and  $\psi_i\in \C$ for  $1 \leq i \leq m$, then  $f_i=F_i^1\circ \psi_i=e^{-\psi_i}$ and $F_i^2 \circ\psi_i^*=e^{-\psi_i^*}=f_i^{\circ}$  are log-concave functions. Therefore, $as^{orlicz}_{\vec h} (\vec f)$,  the mixed Orlicz affine surface area of $\vec f=(f_1, \cdots, f_m)$ can be formulated as   $$as^{orlicz}_{\vec h} (\vec f)= as^{orlicz}_{\vec h,  (e^{-t}, \cdots, e^{-t}), (e^{-t}, \cdots, e^{-t})}(\vec \psi).$$ It is a  non-homogeneous extension of the mixed $L_p$ affine surface area of log-concave functions \cite{CaglarWerner2}.  Similarly, one can define  $G^{orlicz}_{\vec h} (\vec f)$,  the mixed Orlicz geominimal surface area of $\vec f$ by $$G^{orlicz}_{\vec h} (\vec f)= G^{orlicz}_{\vec h,  (e^{-t}, \cdots, e^{-t}), (e^{-t}, \cdots, e^{-t})}(\vec \psi).$$   
    
 The general  mixed Orlicz  affine and geominimal surface areas for multiple convex functions are  $SL_{\pm}(n)$-invariant. That is, for all $T\in SL_{\pm}(n)$ and $\vec h \in \Phi^m  \cup \Psi^m $,  
$$\OrliczAmix \big(\vec\psi\circ T\big)=\OrliczAmix(\vec \psi),\ \ \ \  \OrliczGmix \big(\vec\psi\circ T\big)=\OrliczGmix(\vec \psi)$$  where $ \vec\psi\circ T =  \left( \psi_1\circ T, \cdots, \psi_m \circ T \right)$. In particular, $as^{orlicz}_{\vec h}(\vec f)$ and $G^{orlicz}_{\vec h}(\vec f)$ are $SL_{\pm}(n)$-invariant.    
   
A direct consequence of H\"{o}lder's inequality is the following Alexander-Fenchel type inequality for the general  mixed Orlicz  affine and geominimal surface areas for multiple  convex functions. Note that the classical Alexandrov-Fenchel inequality for mixed volumes of convex bodies is one of the key inequalities in convex geometry with many applications (see e.g., \cite{SchneiderBook}). 
\begin{theo}\label{A-F}
\noindent   Let $\vec h\in \Phi^m\cup\Psi^m$ and $F_i^1, F_i^2 \colon \R\to (0, \infty)$ for all $1\leq i\leq m$. Then
\begin{eqnarray*} 
\left[ \OrliczAmix(\vec\psi) \right]^m  \ \leq  \  \prod_{k=1}^m \OrliczAk(\psi_k),\ \ \ \ \ \left[ \OrliczGmix(\vec\psi) \right]^m  \ \leq  \  \prod_{k=1}^m \OrliczGk(\psi_k).
\end{eqnarray*} 
Moreover, if $\vec h \in \Psi ^m$, one has, for all $1\leq r\leq m-1$, 
\begin{eqnarray*} 
\left[ \OrliczAmix(\vec\psi) \right]^r \leq   \prod_{ k = m-r+1}^{m} \OrliczAmixrk(\vec \psi_{r, k}), \ \ \ \ \ \left[ \OrliczGmix(\vec\psi) \right]^r \leq   \prod_{ k = m-r+1}^{m} \OrliczGmixrk(\vec \psi_{r, k}), \end{eqnarray*}  where $ \vec{F}^1_{r,k}, \vec{F}^2_{r,k}, \vec h_{r,k}$ and $\vec \psi_{r, k}$ are defined similarly with the following form:  for  $1\leq r\leq m-1$ and $ m-r<k\leq m$,   $ \vec{h}_{r,k}= (h_1, h_2, \cdots , h_{m-r}, \underbrace{h_{k}, \cdots , h_{k}}_r)$.  \end{theo} 
       For $1 \leq i \leq m$, let  $$\hat{c}_i=\bigg(\frac{I(\breve{F_i}, 1 )}{I(F_i^2 \circ \psi_i^*, \psi_i^*)}\bigg)^{\frac{1}{n}}   \ \ \mathrm{and} \ \ \bar{c}_i=\bigg(\frac{I(\breve{F_i}, 1)}{I(F_i^1 \circ\psi_i, \psi_i)}\bigg)^{\frac{1}{n}},$$ 
   where, for $F_i^1, F_i^2: \bbR\rightarrow (0, \infty)$, the  decreasing function  $\breve{F_i}:\R \to (0, \infty)$  is defined  by 
   \begin{equation*}
   \breve{F_i}(t)= \sup_{\frac{t_1+ t_2}{2} \ge t} \sqrt{F_i^1(t_1) F_i^2 ( t_2 )}.   
   \end{equation*} 
   
The following functional isoperimetric inequality is a direct consequence of  Theorems \ref{generalaffineineq} and \ref{A-F}. \begin{cor} \label{cor--6} Let  $\psi_i \in \C_0$ and $F_i^1, F_i^2 \colon \R \to(0, \infty)$ be  such that  $0<I( \breve{F_i}, 1)<\infty$ for all $1\leq i\leq m$. 
    
     \noindent  (i)  Let $\vec h\in \Phi^m$. If $0<I(F_i^2 \circ \psi_i^*, \psi^*_i)<\infty$ for all $1\leq i\leq m$, one has, 
     \[   \left[ \OrliczAmix(\vec\psi) \right]^m \ \leq   \ \prod_{i=1}^m \OrliczAi\Big(\frac{\|\cdot \|^2}{2\cdot \hat{c}_i^2}\Big),\] and if in addition $F_i^2 \circ \psi_i^*$ and $\breve{F_i} (\frac{\|\cdot  \|^2}{2})$ are log-concave for all $1\leq i\leq m$,  \[   \left[ \OrliczGmix(\vec\psi) \right]^m \ \leq   \ \prod_{i=1}^m \OrliczGi\Big(\frac{\|\cdot \|^2}{2\cdot \hat{c}_i^2}\Big).\]
    (ii)  Let $\vec{h} \in \Phi^m$ with each $h_i$ being  decreasing. If  $0<I(F_i^1 \circ \psi_i, \psi_i)<\infty$ for all $1\leq i\leq m$, one has  
   \[ \left[ \OrliczAmix(\vec\psi) \right]^m \ \leq  \ \prod_{i=1}^m  \OrliczAi\Big(\frac{ \bar{c}_i^2\cdot \|\cdot \|^2}{2}\Big),\] and if in addition $F_i^2 \circ \psi_i^*$ and $\breve{F_i} (\frac{\|\cdot \|^2}{2})$ are log-concave for all $1\leq i\leq m$,  \[   \left[ \OrliczGmix(\vec\psi) \right]^m \ \leq   \ \prod_{i=1}^m \OrliczGi\Big(\frac{ \bar{c}_i^2\cdot \|\cdot \|^2}{2}\Big).\]
    \end{cor}
   
    In a similar manner, one can define the $i$-th general  mixed Orlicz  affine and geominimal surface areas for two convex functions.   Hereafter,   let vectors $ \vec h, \vec{\psi}, \vec g , \vec F^1$ and $\vec F^2$ be as above, but with only $2$ coordinates.  Assume that $X_{\vec{\psi}}=X_{\psi_1} \cap X_{\psi_2}$ is a nonempty set. Define  
   \begin{eqnarray*}
   V_{\vec{h}, i, \vec F^1, \vec F^2} (\vec \psi, \vec g) 
   &=& \int_{X_{\vec{\psi}}}  \bigg[ h_1 \bigg(\frac{g_1 (\nabla \psi_1 (x))} {F_1^2 ( \langle x, \nabla \psi_1(x) \rangle - \psi_1(x))}\bigg) F_1^1 (\psi_1(x)) \bigg]^{\frac{n-i}{n}} \\
   && \hskip 1cm  \times \bigg[ h_2 \bigg(\frac{g_2 (\nabla \psi_2 (x))} {F_2^2 ( \langle x, \nabla \psi_2(x) \rangle - \psi_2(x))} \bigg) F_2^1 (\psi_2(x)) \bigg]^{\frac{i}{n}}\,dx.  
   \end{eqnarray*}  
   
   We can define the $i$-th general  mixed Orlicz affine  and geominimal surface areas for $\vec \psi$ as follows. 
   \begin{defn} \label{i-th Mixed Orlicz affine surface}  Let   $\psi_i \in \C$ and $F_i^1, F_i^2 \colon \R\to (0, \infty)$ be measurable functions  for $i=1,2$.    For $\vec h \in \Phi^2$,   define  the $i$-th general  mixed Orlicz affine surface area  for $\vec \psi$ by 
   \begin{eqnarray*}
     \OrliczAmixith(\vec\psi) = \inf \left\{V_{ \vec h, i, \vec F^1, \vec F^2} (\vec \psi, \vec g):  g_i \in \mathcal{F}^+_{\psi_i^*} \ \mathrm{with} \ I(g_1, \psi_1^*)= I(g_2, \psi_2^*)=(\sqrt{2\pi})^n \right\},
     \end{eqnarray*} and  define  the $i$-th general  mixed Orlicz geominimal surface area  for $\vec \psi$ by 
   \begin{eqnarray*}
     \OrliczGmixith(\vec\psi) = \inf \left\{V_{ \vec h, i, \vec F^1, \vec F^2} (\vec \psi, \vec g):  g_i \in \mathcal{L}_{\psi_i^*} \ \mathrm{with} \ I(g_1, \psi_1^*)= I(g_2, \psi_2^*)=(\sqrt{2\pi})^n \right\}. 
     \end{eqnarray*} 
    For $\vec h\in \Psi^2$,    $\OrliczAmixith(\vec \psi)$ and $ \OrliczGmixith(\vec\psi)$ can be defined similarly, with `` $\inf$" replaced by `` $\sup$". 
   \end{defn}

    Let  $i < j < k$.  For $ h_1, h_2 \in \Psi$,   H\"older's  inequality implies   \[   \left[ \OrliczAmixjth(\vec \psi) \right]^{k-i}  \ \leq  \ \left[ \OrliczAmixith(\vec \psi) \right]^{k-j}  \left[ \OrliczAmixkth(\vec \psi) \right]^{j-i} .\] This inequality  (with $i=0$ and $k=n$) together with Thoerem \ref{generalaffineineq}  imply, for instance, the following  isoperimetric inequality:  for $0 < j < n$,
   \[  \big[ \OrliczAmixjth(\vec \psi)  \big]^{n}  \ \leq \  \bigg[as^{orlicz} _{h_1, \breve{F_1}, \breve{F_1}} \Big(\frac{\|\cdot \|^2}{2\cdot \hat{c}_1^2}\Big)\bigg]^{n-j}   \bigg[as^{orlicz}_{h_2, \breve{F_2}, \breve{F_2}} \Big(\frac{\|\cdot \|^2}{2\cdot \hat{c}_2^2}\Big)   \bigg]^{j}, \] if $\vec h, \vec F^1, \vec F^2$ and $\vec \psi$  satisfy the same conditions as those for part (i) in Corollary \ref{cor--6};  while if they satisfy the same conditions as those for part (ii) in Corollary \ref{cor--6},  then 
    \[  \big[ \OrliczAmixjth(\vec \psi)\big ]^{n}  \ \leq \   \bigg [as^{orlicz}_{h_1, \breve{F_1}, \breve{F_1}} \Big(\frac{\bar{c}_1^2\cdot \|\cdot \|^2}{2}\Big) \bigg]^{n-j}  \bigg [as^{orlicz}_{h_2, \breve{F_2}, \breve{F_2}}\Big(\frac{\bar{c}_2^2 \cdot \|\cdot \|^2}{2}\Big) \bigg]^{j}. \]  
    Similar inequalities hold for $\OrliczGmixith(\vec\psi)$ as long as corresponding conditions verified. 
   
   \vskip 2mm \noindent {\bf Acknowledgments.} The research of DY is supported
 by a NSERC grant.  The authors are greatly indebted to the referee for many valuable comments which improve largely the quality of the paper.

\vskip 5mm 
\noindent 
Umut Caglar, \ \ \ {\small \tt ucaglar@fiu.edu}
\\
{\small \em Department of Mathematics and Statistics} \\
{\small \em Florida International University} \\
{\small \em Miami, FL 33199, U. S. A. }

\vskip 2mm \noindent Deping Ye, \ \ \ {\small \tt deping.ye@mun.ca}\\
{\small \em Department of Mathematics and Statistics\\
   Memorial University of Newfoundland\\
   St. John's, Newfoundland, Canada A1C 5S7 }
\end{document}